\documentclass[12pt]{amsart}
\usepackage{amscd}
\usepackage{amsfonts,amssymb,amsmath,amsthm}
\usepackage{url}
\usepackage{color}
\usepackage{enumerate}
\usepackage{bm}
\usepackage{geometry}
\usepackage[all]{xy}
\usepackage[new]{old-arrows}

\newtheorem{thm}{Theorem}[section]
\newtheorem{cor}[thm]{Corollary}
\newtheorem{lem}[thm]{Lemma}
\newtheorem{prop}[thm]{Proposition}

\theoremstyle{definition}
\newtheorem{defin}[thm]{Definition}
\newtheorem{rem}[thm]{Remark}

\numberwithin{equation}{section}

\makeatletter
\@namedef{subjclassname@2020}{%
  \textup{2020} Mathematics Subject Classification}
\makeatother

\author[Y. Shiba]{Yuta Shiba}
\address{
Department of Mathematics, 
Graduate School of Science, Tokyo University of Science\\
1-3 Kagurazaka, Shinjuku-ku, Tokyo, 162-8601, Japan}
\email{1120507@ed.tus.ac.jp}
\author[K. Sanada]{Katsunori Sanada}
\address{
Department of Mathematics, 
Faculty of Science, Tokyo University of Science\\
1-3 Kagurazaka, Shinjuku-ku, Tokyo, 162-8601, Japan}
\email{sanada@rs.tus.ac.jp}
\author[A. Itaba]{Ayako Itaba}
\address{
Institute of Arts and Sciences, 
Tokyo University of Science\\
6-3-1 Niijuku, Katsushika-ku, Tokyo, 125-8585, Japan}
\email{itaba@rs.tus.ac.jp}


\newcommand{\map}[3]{{#1}\colon {#2} \longrightarrow {#3}}
\newcommand{\Hom}[3]{\mathrm{Hom}_{#1} ({#2},{#3})}

\newcommand{\coho}[2]{\mathrm{H}^{#1}({#2})}
\newcommand{\hoch}[3]{\mathrm{HH}^{#1}({#2},{#3})}
\newcommand{\shoch}[3]{\mathrm{HHS}^{#1}({#2},{#3})}
\newcommand{\hcoho}[3]{\mathrm{H}^{#1}({#2},{#3})}
\newcommand{\shcoho}[3]{\mathrm{HS}^{#1}({#2},{#3})}
\newcommand{\res}[4]{\mathrm{\mathrm{#1}}_{#2}^{\mathrm{#3}}({#4})}
\newcommand{\sym}[3]{\mathrm{\widetilde{S}}_{#1}^{\mathrm{#2}}({#3})}
\newcommand{\cpx}[5]{\mathrm{#1}_{\mathrm{#2}}^{#3}({#4},{#5})}
\newcommand{\diff}[2]{d_{#1}^{\mathrm{#2}}}
\newcommand{\dif}[2]{\delta_{\mathrm{#1}}^{#2}}

\newcommand{\ch}[1]{\mathrm{ch}({#1})}
\newcommand{\orb}[2]{\mathrm{Orb}_{#1}(#2)}
\newcommand{\ext}[4]{\mathrm{Ext}_{#1}^{#2}({#3},{#4})}

\newcommand{\ad}{\mathrm{ad}}
\newcommand{\tw}{\mathrm{tw}}
\newcommand{\id}[1]{\mathrm{id}_{#1}}
\newcommand{\op}{\mathrm{op}}
\newcommand{\e}{\mathrm{e}}
\newcommand{\Z}{\mathbb{Z}}

\begin{document}

\title[Symmetric cohomology and symmetric Hochschild cohomology]
{Symmetric cohomology and symmetric Hochschild cohomology 
of cocommutative Hopf algebras}

\begin{abstract}
Staic defined symmetric cohomology of groups and studied 
that the secondary symmetric cohomology group is corresponding 
to group extensions and the injectivity of the canonical map 
from symmetric cohomology to classical cohomology. 
In this paper, we define symmetric cohomology 
and symmetric Hochschild cohomology for cocommutative Hopf algebras. 
The first one is a generalization of symmetric cohomology of groups. 
We give an isomorphism between symmetric cohomology 
and symmetric Hochschild cohomology, 
which is a symmetric version of the classical result 
about cohomology of groups by Eilenberg-MacLane 
and cohomology of Hopf algebras 
by Ginzburg-Kumar.  
Moreover, to consider the condition that symmetric cohomology coincides 
with classical cohomology, 
we investigate the projectivity of a resolution which gives symmetric cohomology.
\end{abstract}

\maketitle


\section{Introduction}
Let $G$ be a group and $X$ a $G$-module. 
Motivated by topological geometry, 
Staic \cite{staic20093} defined symmetric cohomology of groups 
by constructing an action of the symmetric group $S_{\bullet +1}$ 
on the standard resolution $\cpx{C}{}{\bullet}{G}{X}$ 
which gives the group cohomology $\hcoho{\bullet}{G}{X}$. 
That is, by taking cohomology of the subcomplex 
$\cpx{CS}{}{\bullet}{G}{X} = \cpx{C}{}{\bullet}{G}{X}^{S_{\bullet+1}}$ 
of $\cpx{C}{}{\bullet}{G}{X}$ fixed by $S_{\bullet +1}$, 
Staic defined the symmetric cohomology 
$\shcoho{\bullet}{G}{X} = \coho{\bullet}{\cpx{CS}{}{\ast}{G}{X}}$. 
In the same paper, for a topological space $U$ 
and the $i$-th homotopy group $\pi_i(U)$, 
Staic proved that, if $\pi_1(U)$ has no elements of order $2$ and $3$, 
then the image of $\alpha \in \shcoho{3}{\pi_1(U)}{\pi_2(U)}$ 
in $\hcoho{3}{\pi_1(U)}{\pi_2(U)}$ 
by the canonical map is Postnikov invariant. 
In \cite{staic2009symmetric}, it was proved 
that the secondary cohomology group $\shcoho{2}{G}{X}$ 
is corresponding to extensions of groups 
which satisfies some conditions. 
Moreover, Staic studied the injectivity of the canonical map 
$\shcoho{\bullet}{G}{X} \longrightarrow \hcoho{\bullet}{G}{X}$ 
induced by the inclusion 
$\cpx{CS}{}{\bullet}{G}{X} \longhookrightarrow \cpx{C}{}{\bullet}{G}{X}$. 
Singh \cite{singh2013symmetric} defined the symmetric continuous 
cohomology of topological groups and the symmetric smooth cohomology 
of Lie groups. 
Recently, Coconet-Todea \cite{coconet2021symmetric} defined 
the symmetric Hochschild cohomology of twisted group algebras 
which is a generalization of group algebras. 
Our aim of this paper is to study about the symmetric cohomology 
and the symmetric Hochschild cohomology 
for cocommutative Hopf algebras as another generalization of group algebras. 

This paper is organized as follows: 
In Section 2, we recall some properties of Hopf algebras 
and the definition of symmetric cohomology of groups. 
Let $k$ be a field and $A$ a cocommutative Hopf algebra over $k$. 
In Section 3, we define the symmetric cohomology 
$\shcoho{\bullet}{A}{M}$ of $A$ with coefficients 
in any left $A$-module $M$ by constructing an action 
of the symmetric group $S_{\bullet+1}$ on the standard non-homogeneous 
complex $\cpx{C}{}{\bullet}{A}{M}$ 
which gives the Hopf algebra cohomology $\hcoho{\bullet}{A}{M}$. 
That is, this action gives the fixed subcomplex 
$\cpx{CS}{}{\bullet}{A}{M} = \cpx{C}{}{\bullet}{A}{M}^{S_{\bullet +1}}$ 
of $\cpx{C}{}{\bullet}{A}{M}$, and this defines the symmetric cohomology 
$\shcoho{\bullet}{A}{M} = \coho{\bullet}{\cpx{CS}{}{\ast}{A}{M}}$. 
Similarly, we define the symmetric Hochschild cohomology 
$\shoch{\bullet}{A}{M}$ of $A$ with coefficients 
in any $A$-bimodule $M$ by constructing an action 
of the symmetric group $S_{\bullet+1}$ on the standard non-homogeneous 
complex $\cpx{C}{e}{\bullet}{A}{M}$ 
which gives the Hochschild cohomology $\hoch{\bullet}{A}{M}$. 
This action gives the fixed subcomplex 
$\cpx{CS}{e}{\bullet}{A}{M} = \cpx{C}{e}{\bullet}{A}{M}^{S_{\bullet +1}}$ 
of $\cpx{C}{e}{\bullet}{A}{M}$. 
From this, we define the symmetric cohomology 
$\shoch{\bullet}{A}{M} = \coho{\bullet}{\cpx{CS}{e}{\ast}{A}{M}}$. 
In Section 4, first, we consider the resolution of $k$ 
which gives the symmetric cohomology and the resolution of $A$ 
which gives the symmetric Hochschild cohomology. 
Eilenberg and MacLane proved an isomorphism 
between the group cohomology and the Hochschild cohomology 
of group algebras (\cite{eilenberg1947cohomology}, see Theorem \ref{thm.iso-grp-coho}). 
Moreover, its isomorphism was generalized to the case of Hopf algebras 
by Ginzburg and Kumar (\cite{ginzburg1993cohomology}, see Remark \ref{rem.iso-hopf-coho}). 
In Theorem \ref{thm.iso-symcoho-symHoch}, 
we obtain an isomorphism, which is a symmetric version 
of these isomorphisms, 
between the symmetric cohomology $\shcoho{n}{A}{{}^\ad M}$ 
and the symmetric Hochschild cohomology $\shoch{n}{A}{M}$ 
for any $A$-bimodule $M$ and $n \geq 0$, 
where ${}^\ad M$ is a left $A$-module via the left adjoint action. 
Also, similar to the case of symmetric cohomology of groups, 
there is an isomorphism of $k$-vector spaces 
$\shcoho{1}{A}{M} \cong \hcoho{1}{A}{M}$ 
and the canonical map $\shcoho{2}{A}{M} \longrightarrow \hcoho{2}{A}{M}$ 
induced by the inclusion 
$\cpx{CS}{}{\bullet}{A}{M} \longhookrightarrow \cpx{C}{}{\bullet}{A}{M}$ 
is injective. 
Moreover, we obtain a result for the projectivity of the resolution 
of $k$ above in Theorem \ref{thm.res-proj}. 
Finally, we calculate the resolution which gives symmetric cohomology 
of group algebras of cyclic groups of odd prime order.

Throughout the paper, let $k$ be a field, 
and we write $\otimes$ for $\otimes_k$. 
\section{Preliminaries}

In this section, we describe some properties of Hopf algebras 
and the definition of symmetric cohomology of groups.
\subsection{Properties of Hopf algebras}
Let $A$ be a Hopf algebra 
with a coproduct $\map{\Delta}{A}{A \otimes A}$, 
a counit $\map{\varepsilon}{A}{k}$, 
and an antipode $\map{S}{A}{A}$. 
We say that $A$ is \textit{cocommutative} 
if $ \tw \circ \Delta = \Delta$, 
where $\map{\tw}{A \otimes A}{A \otimes A}$ is the morphism given 
by $\tw(a \otimes b) = b \otimes a$ for any $a,b \in A$. 
We will use some standard notation for the coproduct, 
so called \textit{Sweedler notation}; we write
		$\Delta(a) = \sum a^{(1)} \otimes a^{(2)}$,
	where the notation $a^{(1)},a^{(2)}$ for tensor factors is symbolic. 
	Throughout the paper, we omit the summation symbol $\sum$ 
	of Sweedler notation when no confusion occurs. 
	Next, we recall some properties of Hopf algebras. 
\begin{prop}[{\cite[Proposition 4.0.1]{swedler1969hopf}}]
	Let $A$ be a Hopf algebra. Then the followings hold.
	\begin{enumerate}[{\rm (1)}]
		\item $S(ab) = S(b)S(a)$ for any $a,b \in A$,
		\item $S(1_A) = 1_A$,
		\item $\varepsilon \circ S = \varepsilon$,
		\item $(S \otimes S) \circ \tw \circ \Delta = \Delta \circ S$,
		\item If $A$ is commutative or cocommutative, 
		then $S^2 = \id{A}$ holds.
	\end{enumerate}	
\end{prop}

For left $A$-modules $M$ and $N$, 
left $A$-module structures on $M \otimes N$ and $\Hom{k}{M}{N}$ 
are defined as follows referring 
to Witherspoon \cite{witherspoon2019hochschild}.

\begin{defin}[cf. {\cite[Section 9.2]{witherspoon2019hochschild}}]
	Let $A$ be a Hopf algebra, and $M$ and $N$ left $A$-modules.
	\begin{enumerate}[{\rm (1)}]
		\item The $k$-vector space $M \otimes N$ is a left $A$-module via 
					$a \cdot ( m \otimes n) = a^{(1)}m \otimes a^{(2)}n$
				for any $a \in A$, $m \in M$ and $n \in N$.
		\item The $k$-vector space $\Hom{k}{M}{N}$ is a left $A$-module via
					$(a \cdot f)(m) = a^{(1)} f(S(a^{(2)}) m)$
				for any $a \in A$, $m \in M$ and $f \in \Hom{k}{M}{N}$. 
				In particular, if $N = k$, then $\Hom{k}{M}{k}$ 
				has a left $A$-module structure given by
					$(a \cdot f)(m) = f(S(a) m)$
				for any $a \in A$ and $m \in M$, where $k$ is a trivial $A$-module via
					$a \cdot x = \varepsilon (a)x$
				for any $a \in A$ and $x \in k$.
		\item Let ${}^A M$ denote the $A$-submodule of $M$ given by
				$$
					{}^A M = \{ m \in M \mid a \cdot m = \varepsilon(a)m, 
					\text{for any } a\in A \},
				$$
				which is called \textit{the submodule of $A$-invariants of $M$}. 
				Similarly, we denote the submodule of $A$-invariants of $M$ 
				by $M^A$ when $M$ is a right $A$-module.
		\item Let $M$ be an $A$-bimodule. 
		The left action on $M$ is defined by
					$a \cdot m = a^{(1)} m S(a^{(2)})$
				for any $a \in A$ and $m \in M$, 
				which is called \textit{the left adjoint action}. 
				Also, let ${}^{\ad}M$ denote by a left $A$-module $M$ 
				with the left adjoint action.
	\end{enumerate}
\end{defin}
\begin{lem}[cf. {\cite[Lemma 9.2.2]{witherspoon2019hochschild}}]
\label{lem.hopf-hom-invariant}
	Let $A$ be a Hopf algebra, and $M$ and $N$ left $A$-modules. 
	Then there is an isomorphism 
	$\Hom{A}{M}{N} \cong {}^A(\Hom{k}{M}{N})$ as $k$-vector spaces.
\end{lem}

We recall the relationships among $A$-modules 
that are obtained by taking the tensor product 
and the set of homomorphisms.

\begin{lem}[cf. {\cite[Lemma 9.2.5]{witherspoon2019hochschild}}]
\label{lem.hopf-adjunction}
	Let $A$ be a Hopf algebra, 
	and $L,M$ and $N$ left $A$-modules. 
	Then there is a natural isomorphism 
	$\Hom{k}{L \otimes M}{N} \cong \Hom{k}{L}{\Hom{k}{M}{N}}$ 
	as left $A$-modules,
	and a natural isomorphism 
	$\Hom{A}{L \otimes M}{N} \cong \Hom{A}{L}{\Hom{k}{M}{N}}$ 
	as $k$-vector spaces.
\end{lem}

\begin{lem}[cf. {\cite[Lemma 9.2.7]{witherspoon2019hochschild}}]
\label{lem.hopf-hom-dual}
	Let $A$ be a Hopf algebra, 
	and $L$ and $M$ left $A$-modules. 
	If $L$ is finite dimensional as a $k$-vector space, 
	then there is a natural isomorphism 
	$\Hom{k}{L}{M} \cong M \otimes \Hom{k}{L}{k}$ as left $A$-modules.
\end{lem}

The following fact tells us the projectivity of modules over Hopf algebras. 

\begin{lem}[cf. {\cite[Lemma 9.2.9]{witherspoon2019hochschild}}]
\label{lem.hopf-proj}
	Let $P$ be a projective $A$-module and $M$ a left $A$-module. 
	Then $P \otimes M$ is a projective $A$-module. 
	If the antipode $S$ is bijective, 
	then $M \otimes P$ is a projective $A$-module.
\end{lem}
\subsection{Symmetric cohomology of groups}
In this subsection, we recall the definition 
of the symmetric cohomology of groups 
which is introduced by Staic \cite[Section 5]{staic20093}. 
Let $n$ be a non-negative integer, 
$G^n$ direct products of $n$ times of a group $G$, 
when $n=0$, $G^0$ is considered as the trivial group. 
For a $G$-module $X$, we put 
$\cpx{C}{}{n}{G}{X} = \{ \map{f}{G^n}{X} \}$ 
and define $\map{\dif{C}{n}}{\cpx{C}{}{n}{G}{X}}{\cpx{C}{}{n+1}{G}{X}}$ by 
\begin{align*}
	\dif{C}{n}(f)(g_1, \dots ,g_{n+1}) &= g_1 f(g_2,\dots ,g_{n+1}) \\
											 &\quad\,+ \sum_{i=1}^n (-1)^i f(g_1 , \dots , g_i g_{i+1} , \dots ,g_{n+1}) + (-1)^{n+1} f(g_1 , \dots ,g_n).
\end{align*}
Then $\cpx{C}{}{\bullet}{G}{X}$ is a complex of abelian groups. 
Its cohomology is called the \textit{group cohomology} 
and is denoted by $\hcoho{\bullet}{G}{X}$. 
It is constructed an action of the symmetric group $S_{n+1}$ 
on $\cpx{C}{}{n}{G}{X}$ for each $n\geq 0$. 
Namely, the element $\sigma_i = (i,i+1) \in S_{n+1}$ 
acts on $\cpx{C}{}{n}{G}{X}$ by 
\begin{align*}
	(\sigma_1 \cdot f)(g_1 , \dots , g_n) 
	&= -g_1 f(g_1^{-1},g_1g_2,g_3 , \dots ,g_n), \\
	(\sigma_i \cdot f)(g_1 , \dots , g_n) 
	&= -f(g_1,\dots ,g_{i-1}g_i,g_i^{-1},g_ig_{i+1},\dots ,g_n) , 
	\text{ for } 2 \leq i \leq n-1 ,\\
	(\sigma_n \cdot f)(g_1 , \dots , g_n) 
	&= -f(g_1 , \dots ,g_{n-1}g_n , g_n^{-1})
\end{align*}
for $f \in \cpx{C}{}{n}{G}{X}$ and $g_1 , \dots , g_n \in G$. 
Note that this action is compatible with the differential $\dif{C}{}$, 
and hence $\cpx{CS}{}{\bullet}{G}{X} = \cpx{C}{}{\bullet}{G}{X}^{S_{\bullet +1}}$ 
is a subcomplex of $\cpx{C}{}{\bullet}{G}{X}$. 
Its cohomology is called the \textit{symmetric cohomology} 
and is denoted by $\shcoho{\bullet}{G}{X}$. 

On the other hand, $\hcoho{\bullet}{G}{X}$ can be defined alternatively 
using a homogeneous complex. 
For a non-negative integer $n$, and let $\Z [G^{n+1}]$ 
be a $G$-module via
$$
	g \cdot (g_0 , \dots , g_n) = (gg_0, \dots ,gg_n)
$$
for any $g,g_0, \dots , g_n \in G$. 
There is a projective resolution of $\Z$ as $\Z G$-modules:
$$
	\cdots \longrightarrow \Z [G^{n+1}] 
	\overset{\diff{n}{}}{\longrightarrow} \Z [G^{n}] 
	\longrightarrow 
	\cdots \longrightarrow \Z [G^2] 
	\overset{\diff{1}{}}{\longrightarrow} \Z [G^1] 
	\overset{\diff{0}{}}{\longrightarrow} \Z \longrightarrow 0 , 
$$
where we set
$\diff{n}{}(g_0 , \dots , g_{n}) 
= \sum_{i=0}^{n} (-1)^i (g_0 , \dots , g_{i-1} , g_{i+1} , \dots , g_{n})$.
Then $\cpx{K}{}{\bullet}{G}{X}$ and $\cpx{C}{}{\bullet}{G}{X}$ 
are isomorphic as complexes 
where we put $\cpx{K}{}{n}{G}{X} = \Hom{G}{\Z [G^{n+1}]}{X}$ 
and $\dif{K}{n} = \Hom{G}{\diff{n+1}{}}{X}$. Hence $\hcoho{\bullet}{G}{X}$ 
is also defined using $\cpx{K}{}{\bullet}{G}{X}$. 
We define an action of $S_{n+1}$ on $\cpx{K}{}{n}{G}{X}$ 
for each $n \geq 0$. 
Namely, the element $\sigma_i$ acts on $\cpx{K}{}{n}{G}{X}$ by 
$$
	(\sigma_i \cdot f)(g_0 , \dots , g_n) 
	= -f(g_0 , \dots , g_{i} , g_{i-1} , \dots , g_n), \text{ for }1 \leq i \leq n
$$
for $f \in \cpx{K}{}{n}{G}{X}$ and $g_0, \dots , g_n \in G$. 
Note that this action is compatible with the differential $\dif{K}{}$, 
and hence 
$\cpx{KS}{}{\bullet}{G}{X} = \cpx{K}{}{\bullet}{G}{X}^{S_{\bullet +1}}$ 
is a subcomplex of $\cpx{K}{}{\bullet}{G}{X}$. 
Moreover, there is an isomorphism 
$\cpx{CS}{}{\bullet}{G}{X} \cong \cpx{KS}{}{\bullet}{G}{X}$ 
as complexes 
(see Bardakov-Neshchadim-Singh \cite{bardakov2018exterior}, 
Pirashvili \cite{pirashvili2018symmetric}). 
Hence, we see that $\shcoho{\bullet}{G}{X}$ is also defined 
using $\cpx{KS}{}{\bullet}{G}{X}$.
\section{Symmetric cohomology and symmetric Hochschild cohomology}
In this section, we recall the definition 
of the Hopf algebra cohomology 
which is a generalization of the group cohomology 
and define the symmetric cohomology 
for cocommutative Hopf algebras.
\subsection{Definition of symmetric cohomology}
Let us start with the definition of the Hopf algebra cohomology.
\begin{defin}[cf. {\cite[Definition 9.3.5]{witherspoon2019hochschild}}]
	Let $A$ be a Hopf algebra and $M$ a left $A$-module. 
	The \textit{Hopf algebra cohomology} $\hcoho{\bullet}{A}{M}$ 
	of $A$ with coefficients in $M$ is defined by 
    $$
		\hcoho{n}{A}{M} = \ext{A}{n}{k}{M}.
    $$
\end{defin}

We construct a standard non-homogeneous complex 
which gives the Hopf algebra cohomology. 
Let $n$ be a non-negative integer. 
Suppose that, for any $b, a_0, \dots , a_n \in A$, 
$\res{T}{n}{}{A} = A^{\otimes n+1}$ is a left $A$-module via
$
	b \cdot (a_0 \otimes a_1 \otimes \cdots \otimes a_n) 
	= ba_0 \otimes a_1 \otimes \cdots \otimes a_n.
$
Then there is a projective resolution of $k$ as left $A$-modules:
$$
	\cdots \longrightarrow \res{T}{n}{}{A} 
	\overset{\diff{n}{T}}{\longrightarrow} \res{T}{n-1}{}{A} 
	\longrightarrow \cdots \longrightarrow 
	\res{T}{1}{}{A} \overset{\diff{1}{T}}{\longrightarrow} 
	\res{T}{0}{}{A} \overset{\diff{0}{T}}{\longrightarrow} k \longrightarrow 0,
$$
where we set 
\vspace{-0.5em}
\begin{align*}
	&\diff{0}{T} = \varepsilon ,\\
	&\diff{n}{T}(a_0 \otimes \cdots \otimes a_n) 
	= \sum_{i=0}^{n-1} (-1)^i a_0 \otimes \cdots 
	\otimes a_i a_{i+1} \otimes \cdots \otimes a_n \\
	&\hspace{35mm} + (-1)^{n} a_0 \otimes \cdots 
	\otimes a_{n-1} \varepsilon(a_n), \text{ for } n \geq 1.
\end{align*}
Moreover, we denote 
$\cpx{C}{}{n}{A}{M} = \Hom{A}{\res{T}{n}{}{A}}{M}$ 
and $\dif{C}{n} = \Hom{A}{\diff{n+1}{T}}{M}$. 
Next, we define an action of the symmetric group $S_{n+1}$ 
on $\cpx{C}{}{n}{A}{M}$ for each $n \geq 0$. 
Let $A$ be a cocommutative Hopf algebra. 
Then we define the action of $\sigma_i = (i,i+1) \in S_{n+1}$ 
on $\cpx{C}{}{n}{A}{M}$ by 
\begin{align}
\label{eq:sa1}
	\begin{cases}
		(\sigma_i f)(a_0 \otimes \cdots \otimes a_n) 
		= -f(a_0 \otimes \cdots \otimes a_{i-1} a_i^{(1)} 
		\otimes S(a_i^{(2)}) \otimes a_i^{(3)} a_{i+1} \otimes 
		\cdots  \otimes a_n), \\
		\hspace{105.5mm}\text{ for } 1 \leq i \leq n-1, \\
		(\sigma_n f)(a_0 \otimes \cdots \otimes a_n) 
		= -f(a_0 \otimes \cdots \otimes a_{n-1}a_n^{(1)} \otimes S(a_n^{(2)}))
	\end{cases}
\end{align}
for $f \in \cpx{C}{}{n}{A}{M}$ and $a_0 , \dots , a_n \in A$. 
We show that these formulas \eqref{eq:sa1} are well-defined 
as an action, and the action is compatible with the differential.
\begin{prop}
\label{prop.coho-nonhomo-welldef}
	The above formulas \eqref{eq:sa1} define an action 
	of the symmetric group $S_{n+1}$ on $\cpx{C}{}{n}{A}{M}$ 
	which is compatible with the differential $\dif{C}{n}$ 
	for each $n \geq 0$.
\end{prop}

\begin{proof}
	First, we check that, for $1 \leq i \leq n$, 
	the action of $\sigma_i$ satisfies relations of the Coxeter presentation 
	of $S_{n+1}$ which is 
	\begin{align*}
		S_{n+1} = \langle \sigma_i , i\in \{1,2,\cdots ,n\} \mid \,&\sigma_i^2 
		= e \ \text{for} \ i\in \{1,\cdots , n\} , \\
			&\sigma_i \sigma_{i+1} \sigma_i = \sigma_{i+1} \sigma_i \sigma_{i+1}
			 \ \text{for} \ i \in \{1,\cdots ,n-1\} , \\
			&\sigma_i \sigma_j = \sigma_j \sigma_i\ 
			\text{for}\  |i-j|\geq 2 \rangle. 
	\end{align*}
	Let $f \in \cpx{C}{}{n}{A}{M}$, $a_0,\dots ,a_n \in A$ 
	and $i$ an integer such that $1 \leq i \leq n-1$. 
	For the left hand side of the second relation 
	$\sigma_i \sigma_{i+1} \sigma_i = \sigma_{i+1} \sigma_i \sigma_{i+1}$, 
	we have
	\begin{align*}
		&(\sigma_i \sigma_{i+1} \sigma_i f)(a_0 \otimes \cdots \otimes a_{n}) \\
		&\quad=-(\sigma_{i+1} \sigma_i f)(a_0 \otimes 
		\cdots \otimes a_{i-1} a_i^{(1)} \otimes S(a_i^{(2)}) 
		\otimes a_i^{(3)} a_{i+1} \otimes \cdots \otimes a_{n} ) \\
		&\quad=(\sigma_i f)(a_0 \otimes \cdots 
		\otimes a_{i-1} a_i^{(1)} \otimes S(a_i^{(2)})a_i^{(3)} a_{i+1}^{(1)} 
		\otimes S(a_i^{(4)} a_{i+1}^{(2)}) \otimes a_i^{(5)} a_{i+1}^{(3)} a_{i+2} 
		\otimes \cdots \otimes a_{n}) \\
		&\quad=(\sigma_i f)(a_0 \otimes \cdots 
		\otimes a_{i-1} a_i^{(1)} \otimes \varepsilon(a_i^{(2)}) a_{i+1}^{(1)} 
		\otimes S(a_{i+1}^{(2)})S(a_i^{(3)}) \otimes a_i^{(4)} a_{i+1}^{(3)} a_{i+2} 
		\otimes \cdots \otimes a_{n}) \\
		&\quad=(\sigma_i f)(a_0 \otimes \cdots \otimes a_{i-1} a_i^{(1)} 
		\otimes a_{i+1}^{(1)} \otimes S(a_{i+1}^{(2)})S(a_i^{(2)}) 
		\otimes a_i^{(3)} a_{i+1}^{(3)} a_{i+2} \otimes \cdots \otimes a_{n}) \\
		&\quad=-f(a_0 \otimes \cdots \otimes a_{i-1} a_i^{(1)} a_{i+1}^{(1)} 
		\otimes S(a_{i+1}^{(2)}) \otimes a_{i+1}^{(3)} S(a_{i+1}^{(4)})S(a_i^{(2)}) 
		\otimes a_i^{(3)} a_{i+1}^{(5)} a_{i+2} \otimes \cdots \otimes a_{n}) \\
		&\quad=-f(a_0 \otimes \cdots \otimes a_{i-1} a_i^{(1)} a_{i+1}^{(1)} 
		\otimes S(a_{i+1}^{(2)}) \otimes \varepsilon(a_{i+1}^{(3)})S(a_i^{(2)}) 
		\otimes a_i^{(3)} a_{i+1}^{(4)} a_{i+2} \otimes \cdots \otimes a_{n}) \\
		&\quad=-f(a_0 \otimes \cdots \otimes a_{i-1} a_i^{(1)} a_{i+1}^{(1)} 
		\otimes S(a_{i+1}^{(2)}) \otimes S(a_i^{(2)}) \otimes a_i^{(3)}  
		\varepsilon(a_{i+1}^{(3)}) a_{i+1}^{(4)} a_{i+2} 
		\otimes \cdots \otimes a_{n}) \\
		&\quad=-f(a_0 \otimes \cdots \otimes a_{i-1} a_i^{(1)} a_{i+1}^{(1)} 
		\otimes S(a_{i+1}^{(2)}) \otimes S(a_i^{(2)}) 
		\otimes a_i^{(3)} a_{i+1}^{(3)} a_{i+2} \otimes \cdots \otimes a_{n}).
	\end{align*}
On the other hand, for the right hand side, we have
\vspace{-0.5em}
	\begin{align*}
		&(\sigma_{i+1} \sigma_i \sigma_{i+1} f)(a_0 \otimes \cdots \otimes a_{n}) \\
		&\quad=-(\sigma_i \sigma_{i+1} f)(a_0 \otimes \cdots  
		\otimes a_i a_{i+1}^{(1)} \otimes S(a_{i+1}^{(2)}) 
		\otimes a_{i+1}^{(3)} a_{i+2} \otimes \cdots \otimes a_{n}) \\
		&\quad=(\sigma_{i+1} f)(a_0 \otimes \cdots 
		\otimes a_{i-1} a_i^{(1)} a_{i+1}^{(1)} \otimes S(a_i^{(2)} a_{i+1}^{(2)}) 
		\otimes a_i^{(3)} a_{i+1}^{(3)} S(a_{i+1}^{(4)}) 
		\otimes a_{i+1}^{(5)} a_{i+2} \otimes \cdots \otimes a_{n}) \\
		&\quad=(\sigma_{i+1} f)(a_0 \otimes \cdots 
		\otimes a_{i-1} a_i^{(1)} a_{i+1}^{(1)} \otimes  S(a_{i+1}^{(2)})S(a_i^{(2)}) 
		\otimes a_i^{(3)} \varepsilon(a_{i+1}^{(3)}) \otimes a_{i+1}^{(4)} a_{i+2} 
		\otimes \cdots \otimes a_{n}) \\
		&\quad=(\sigma_{i+1} f)(a_0 \otimes \cdots 
		\otimes a_{i-1} a_i^{(1)} a_{i+1}^{(1)} \otimes 
		S(a_{i+1}^{(2)})S(a_i^{(2)}) \otimes a_i^{(3)} \otimes a_{i+1}^{(3)} a_{i+2} 
		\otimes \cdots \otimes a_{n}) \\ 
		&\quad= -f(a_0 \otimes \cdots 
		\otimes a_{i-1} a_i^{(1)} a_{i+1}^{(1)} 
		\otimes S(a_{i+1}^{(2)})S(a_i^{(2)}) a_i^{(3)} 
		\otimes S(a_i^{(4)}) \otimes a_i^{(5)} a_{i+1}^{(3)} a_{i+2} 
		\otimes \cdots \otimes a_{n}) \\
		&\quad= -f(a_0 \otimes \cdots \otimes a_{i-1} a_i^{(1)} a_{i+1}^{(1)} 
		\otimes S(a_{i+1}^{(2)})\varepsilon(a_i^{(2)}) 
		\otimes S(a_i^{(3)}) \otimes a_i^{(4)} a_{i+1}^{(3)} a_{i+2} 
		\otimes \cdots \otimes a_{n}) \\
		&\quad= -f(a_0 \otimes \cdots \otimes a_{i-1} a_i^{(1)} a_{i+1}^{(1)} 
		\otimes S(a_{i+1}^{(2)}) \otimes S(a_i^{(2)}) 
		\otimes a_i^{(3)} a_{i+1}^{(3)} a_{i+2} \otimes \cdots \otimes a_{n}). 
	\end{align*}
Hence, $\sigma_i \sigma_{i+1} \sigma_i f = \sigma_{i+1} \sigma_i \sigma_{i+1} f$ holds. 
Similarly, the other relations can be checked by calculations.
	
Next, we show that this action is compatible 
with the differential $\dif{C}{n}$. 
Let $f\in \cpx{C}{}{n}{A}{M}$ be invariant under the action of $S_{n+1}$ and $i$ an integer such that 
$3 \leq i \leq n-2$. 
We have
	\vspace{-0.5em}	
	\begin{align*}
		&(\sigma_i \dif{C}{n}(f))(a_0 \otimes \cdots \otimes a_{n+1}) \\
		&\quad= -\dif{C}{n}(f)(a_0 \otimes \cdots \otimes a_{i-1} a_i^{(1)} 
		\otimes S(a_i^{(2)}) \otimes a_i^{(3)} a_{i+1} \otimes \cdots 
		\otimes a_{n+1}) \\
		&\quad= -\Bigg\{ \sum_{j=0}^{i-3} (-1)^j 
		f(a_0 \otimes \cdots \otimes a_j a_{j+1} \otimes \cdots 
		\otimes a_{i-1} a_i^{(1)} \otimes S(a_i^{(2)}) \otimes a_i^{(3)} a_{i+1} 
		\otimes \cdots \otimes a_{n+1}) \\
		&\quad\quad\quad\quad\,+(-1)^{i-2} f(a_0 \otimes \cdots 
		\otimes a_{i-2} a_{i-1} a_i^{(1)} \otimes S(a_i^{(2)}) 
		\otimes a_i^{(3)} a_{i+1} \otimes \cdots \otimes a_{n+1}) \\
		&\quad\quad\quad\quad\,+(-1)^{i-1} 
		f(a_0 \otimes \cdots \otimes a_{i-1} a_i^{(1)}S(a_i^{(2)}) 
		\otimes a_i^{(3)} a_{i+1} \otimes \cdots \otimes a_{n+1}) \\
		&\quad\quad\quad\quad\,+(-1)^i f(a_0 \otimes \cdots 
		\otimes a_{i-1}a_i^{(1)} \otimes S(a_i^{(2)}) a_i^{(3)} a_{i+1} 
		\otimes \cdots \otimes a_{n+1}) \\
		&\quad\quad\quad\quad\,+(-1)^{i+1} f(a_0 \otimes \cdots 
		\otimes a_{i-1} a_i^{(1)} \otimes S(a_i^{(2)}) 
		\otimes a_i^{(3)} a_{i+1} a_{i+2} \otimes \cdots \otimes a_{n+1}) \\
		&\quad\quad\quad\quad\, +\sum_{j=i+2}^{n}(-1)^j 
		f(a_0 \otimes \cdots \otimes a_{i-1} a_i^{(1)} \otimes S(a_i^{(2)}) 
		\otimes a_i^{(3)} a_{i+1} \otimes \cdots \otimes a_j a_{j+1} 
		\otimes \cdots \otimes a_{n+1}) \\
		&\quad\quad\quad\quad\,+(-1)^{n+1} f(a_0 \otimes \cdots 
		\otimes a_{i-1} a_i^{(1)} \otimes S(a_i^{(2)}) \otimes a_i^{(3)} a_{i+1} 
		\otimes \cdots \otimes a_n \varepsilon(a_{n+1})) \Bigg\}.
	\end{align*}
	By deforming the third and fourth term of the above formula, we have
	\begin{align*}
		&(-1)^{i-1} f(a_0 \otimes \cdots \otimes a_{i-1} a_i^{(1)}S(a_i^{(2)}) 
		\otimes a_i^{(3)} a_{i+1} \otimes \cdots \otimes a_{n+1}) \\
		&\quad= (-1)^{i-1} f(a_0 \otimes \cdots \otimes a_{i-1} 
		\varepsilon(a_i^{(1)}) \otimes a_i^{(2)} a_{i+1} \otimes \cdots 
		\otimes a_{n+1}) \\
		&\quad= (-1)^{i-1} f(a_0 \otimes \cdots \otimes a_{i-1} 
		\otimes a_i a_{i+1} \otimes \cdots \otimes a_{n+1}) ,\\
		&(-1)^i f(a_0 \otimes \cdots \otimes a_{i-1}a_i^{(1)} 
		\otimes S(a_i^{(2)}) a_i^{(3)} a_{i+1} \otimes \cdots \otimes a_{n+1}) \\
		&\quad=(-1)^i f(a_0 \otimes \cdots \otimes a_{i-1}a_i^{(1)} 
		\otimes \varepsilon(a_i^{(2)})  a_{i+1} \otimes 
		\cdots \otimes a_{n+1}) \\
		&\quad=(-1)^i f(a_0 \otimes \cdots \otimes a_{i-1}a_i 
		\otimes a_{i+1} \otimes \cdots \otimes a_{n+1}). 
	\end{align*}
	Moreover, using that $f$ is invariant under the action of $S_{n+1}$, 
	the first term is deformed by
	\begin{align*}
		&f(a_0 \otimes \cdots \otimes a_j a_{j+1} \otimes \cdots 
		\otimes a_{i-1} a_i^{(1)} \otimes S(a_i^{(2)}) \otimes a_i^{(3)} a_{i+1} 
		\otimes \cdots \otimes a_{n+1}) \\
		&\quad= -(\sigma_i f)(a_0 \otimes \cdots \otimes a_j a_{j+1} 
		\otimes \cdots \otimes a_{i-1} \otimes a_i \otimes a_{i+1} 
		\otimes \cdots \otimes a_{n+1}) \\
		&\quad= -f(a_0 \otimes \cdots \otimes a_j a_{j+1} \otimes 
		\cdots \otimes a_{i-1} \otimes a_i \otimes a_{i+1} \otimes \cdots \otimes a_{n+1}).
	\end{align*}
	Similarly, other terms can be deformed. 
	Hence, for $3 \leq i \leq n-2$, 
	$\sigma_i \dif{C}{n}(f) = \dif{C}{n}(f)$ holds. 
	By same as the above calculations, 
	the same equation holds for other cases.
\end{proof}

According to Proposition \ref{prop.coho-nonhomo-welldef}, 
the sequence of invariant subspaces 
$\cpx{C}{}{\bullet}{A}{M}^{S_{\bullet +1}}$ 
by the action of symmetric groups becomes a subcomplex 
of $\cpx{C}{}{\bullet}{A}{M}$ and is denoted by $\cpx{CS}{}{\bullet}{A}{M}$. 

\begin{defin}
\label{def.sym-coho}
	Let $A$ be a cocommutative Hopf algebra 
	and $M$ a left $A$-module. 
	The \textit{symmetric cohomology} $\shcoho{\bullet}{A}{M}$ of $A$ 
	with coefficients in $M$ is defined by 
    $$
		\shcoho{n}{A}{M} = \coho{n}{\cpx{CS}{}{\bullet}{A}{M}}.
	$$
\end{defin}

\begin{rem}
	By the structure of group algebras as Hopf algebras, 
	we note that the symmetric cohomology of cocommutive Hopf algebras 
	in Definition \ref{def.sym-coho} 
	is a generalization of the symmetric cohomology of groups.
\end{rem}
\subsection{Symmetric cohomology constructed by the homogeneous complex}
Similar to group cohomology, 
we can construct a homogeneous complex 
which gives the Hopf algebra cohomology. 
We describe its construction and an action of symmetric groups 
on the homogeneous complex. 
Let $A$ be a Hopf algebra and $n$ a non-negative integer. 
Then $\res{\widetilde{T}}{n}{}{A} = A^{\otimes n+1}$ 
is a left $A$-module via
$$
	 b \cdot (a_0 \otimes a_1 \otimes \cdots \otimes a_n) 
	 = b^{(1)} a_0 \otimes b^{(2)} a_1 \otimes \cdots \otimes b^{(n+1)} a_n
$$
for any $b, a_0, \dots , a_n \in A$, 
and there is a projective resolution of $k$ as left $A$-modules:
\begin{align*}
	\cdots \longrightarrow \res{\widetilde{T}}{n}{}{A} 
	\overset{\diff{n}{\widetilde{T}}}{\longrightarrow} 
	\res{\widetilde{T}}{n-1}{}{A} \longrightarrow 
	\cdots \longrightarrow \res{\widetilde{T}}{1}{}{A} 
	\overset{\diff{1}{\widetilde{T}}}{\longrightarrow} 
	\res{\widetilde{T}}{0}{}{A} 
	\overset{\diff{0}{\widetilde{T}}}{\longrightarrow} k \longrightarrow 0,
\end{align*}
where we set
\begin{align*}
	&\diff{0}{\widetilde{T}} = \varepsilon , \\
	&\diff{n}{\widetilde{T}}(a_0 \otimes \cdots \otimes a_n) 
	= \sum_{i=0}^{n}(-1)^i \varepsilon(a_i) a_0 \otimes 
	\cdots \otimes a_{i-1} \otimes a_{i+1} \otimes \cdots 
	\otimes a_n , \text{ for } n \geq 1.
\end{align*}
Moreover, we denote 
$\cpx{K}{}{n}{A}{M} = \Hom{A}{\res{\widetilde{T}}{n}{}{A}}{M}$ 
and $\dif{K}{n} = \Hom{A}{\diff{n+1}{\widetilde{T}}}{M}$. 
Next, we define an action of $S_{n+1}$ on $\cpx{K}{}{n}{A}{M}$ 
for each $n \geq 0$. 
Let $A$ be a cocommutative Hopf algebra. 
Then, for $f \in \cpx{K}{}{n}{A}{M}$, $a_0, \dots , a_n \in A$ 
and $1 \leq i \leq n$, 
we define the action of $\sigma_i$ on $\cpx{K}{}{n}{A}{M}$ by 
\begin{align}
\label{eq:sa2}
	(\sigma_i \cdot f)(a_0 \otimes \cdots \otimes a_n) 
	= - f(a_0 \otimes \cdots \otimes a_i \otimes a_{i-1} 
	\otimes \cdots \otimes a_n).
\end{align}
We show that these formulas \eqref{eq:sa2} 
are well-defined as an action, and the action is compatible with the differential.
\begin{prop}
\label{prop.coho-homo-welldef}
	The above formulas \eqref{eq:sa2} define an action 
	of the symmetric group $S_{n+1}$ on $\cpx{K}{}{n}{A}{M}$ 
	which is compatible with the differential $\dif{K}{n}$.
\end{prop}

\begin{proof}
	First, we check that, for $1 \leq i \leq n$, 
	the action of $\sigma_i$ satisfies relations of the Coxeter presentation 
	of $S_{n+1}$. 
	Let $f \in \cpx{K}{}{n}{A}{M}$, $a_0,\dots ,a_n \in A$ 
	and $i$ an integer such that $1 \leq i \leq n-1$. 
	For the second relation 
	$\sigma_i \sigma_{i+1} \sigma_i = \sigma_{i+1} \sigma_i \sigma_{i+1}$, 
	we have
	\begin{align*}
		(\sigma_i \sigma_{i+1} \sigma_i f)(a_0 \otimes \cdots \otimes a_{n}) 
		&=-(\sigma_{i+1} \sigma_i f)(a_0 \otimes \cdots \otimes a_i 
		\otimes a_{i-1} \otimes \cdots \otimes a_{n}) \\
		&=(\sigma_i f)(a_0 \otimes \cdots \otimes a_i \otimes a_{i+1} 
		\otimes a_{i-1} \otimes \cdots \otimes a_{n}) \\
		&=-f(a_0 \otimes \cdots \otimes a_{i+1} \otimes a_i 
		\otimes a_{i-1} \otimes \cdots \otimes a_{n}) \\
		&=(\sigma_{i+1} f)(a_0 \otimes \cdots \otimes a_{i+1} 
		\otimes a_{i-1} \otimes a_i \otimes \cdots \otimes a_{n}) \\
		&=-(\sigma_i \sigma_{i+1} f)(a_0 \otimes \cdots \otimes a_{i-1} 
		\otimes a_{i+1} \otimes a_i \otimes \cdots \otimes a_{n}) \\
		&=(\sigma_{i+1} \sigma_i \sigma_{i+1} f)(a_0 \otimes \cdots 
		\otimes a_{i-1} \otimes a_i \otimes a_{i+1} \otimes \cdots \otimes a_{n})
	\end{align*}
	Hence, $\sigma_i \sigma_{i+1} \sigma_i 
	= \sigma_{i+1} \sigma_i \sigma_{i+1}$ holds. 
	Similarly, the other relations can be checked by calculations. 
	
	Next, we show that this action is compatible 
	with the differential $\dif{K}{n}$. 
	Let $f\in \cpx{K}{}{n}{A}{M}$ be invariant under the action of $S_{n+1}$ and $i$ an integer 
	such that $3 \leq i \leq n$. 
	Then we have
	\begin{align*}
		&(\sigma_i \dif{K}{n}(f))(a_0 \otimes \cdots \otimes a_{n+1}) \\
		&\quad\,= -\dif{K}{n}(f)(a_0 \otimes \cdots 
		\otimes a_i \otimes a_{i-1} \otimes \cdots \otimes a_{n+1}) \\
		&\quad\,= -\Bigg\{\sum_{j = 0}^{i-3} (-1)^j 
		f(a_0 \otimes \cdots \otimes \varepsilon(a_j)a_{j+1} 
		\otimes \cdots \otimes a_i \otimes a_{i-1} \otimes \cdots 
		\otimes a_{n+1}) \\
		&\quad\quad\quad\quad\,+(-1)^{i-2} 
		f(a_0 \otimes \cdots \otimes \varepsilon(a_{i-2})a_i 
		\otimes a_{i-1} \otimes \cdots \otimes a_{n+1}) \\ 
		&\quad\quad\quad\quad\,+ (-1)^{i-1} 
		f(a_0 \otimes \cdots \otimes \varepsilon(a_i)a_{i-1} 
		\otimes \cdots \otimes a_{n+1}) \\
		&\quad\quad\quad\quad\,+(-1)^i 
		f(a_0 \otimes \cdots  \otimes a_i \otimes 
		\varepsilon(a_{i-1})a_{i+1} \otimes \cdots \otimes a_{n+1}) \\
		&\quad\quad\quad\quad\,+\sum_{j = i+1}^{n+1} (-1)^j 
		f(a_0 \otimes \cdots \otimes a_i \otimes a_{i-1} \otimes \cdots 
		\otimes \varepsilon(a_j)a_{j+1} \otimes \cdots 
		\otimes a_{n+1}) \Bigg\} \\
		&\quad\,= -\sum_{j = 0}^{i-3} (-1)^j 
		f(a_0 \otimes \cdots \otimes \varepsilon(a_j)a_{j+1} \otimes 
		\cdots \otimes a_i \otimes a_{i-1} \otimes \cdots \otimes a_{n+1}) \\
		&\quad\quad\,-(-1)^{i-2} f(a_0 \otimes \cdots \otimes a_i 
		\otimes \varepsilon(a_{i-2})a_{i-1} \otimes \cdots \otimes a_{n+1}) \\ 
		&\quad\quad\,-(-1)^i f(a_0 \otimes \cdots  \otimes 
		\varepsilon(a_{i-1})a_i \otimes a_{i+1} \otimes \cdots \otimes a_{n+1}) \\
		&\quad\quad\,- (-1)^{i-1} 
		f(a_0 \otimes \cdots \otimes a_{i-1} \otimes 
		\varepsilon(a_i)a_{i+1} \otimes \cdots \otimes a_{n+1}) \\
		&\quad\quad\,-\sum_{j = i+1}^{n+1} (-1)^j 
		f(a_0 \otimes \cdots \otimes a_i \otimes a_{i-1} \otimes 
		\cdots \otimes \varepsilon(a_j)a_{j+1} \otimes \cdots \otimes a_{n+1}).
	\end{align*}
	Here, using that $f$ is invariant under the action of $S_{n+1}$, 
	each term can be deformed. 
	Hence, for $3 \leq i \leq n$, $\sigma_i \dif{K}{n}(f) = \dif{K}{n}(f)$ holds. 
	By same as the above calculations, 
	the same equation holds for other cases.
\end{proof}

According to Proposition \ref{prop.coho-homo-welldef}, 
the sequence of invariant subspaces 
$\cpx{K}{}{\bullet}{A}{M}^{S_{\bullet +1}}$ 
by the action of symmetric groups becomes a subcomplex 
of $\cpx{K}{}{\bullet}{A}{M}$ and is denoted by $\cpx{KS}{}{\bullet}{A}{M}$. 
Moreover, we show that there is an isomorphism 
between the non-homogeneous complex and the homogeneous complex.

\begin{prop}
\label{prop.nonhomo-homo-iso}
	Let $A$ be a Hopf algebra and $M$ a left $A$-module. 
	Then there is an isomorphism 
	$\cpx{C}{}{\bullet}{A}{M} \cong \cpx{K}{}{\bullet}{A}{M}$ 
	as complexes. 
	Moreover, if $A$ is cocommutative, 
	then this isomorphism induces an isomorphism 
	$\cpx{CS}{}{\bullet}{A}{M} \cong \cpx{KS}{}{\bullet}{A}{M}$ 
	as complexes. 
	Therefore, $\shcoho{n}{A}{M} = \coho{n}{\cpx{KS}{}{\bullet}{A}{M}}$ holds.
\end{prop}

\begin{proof}
First, we will construct an isomorphism 
$\cpx{C}{}{\bullet}{A}{M} \cong \cpx{K}{}{\bullet}{A}{M}$ 
as complexes. 
For any $f \in \cpx{K}{}{n}{A}{M}$ and $a_0, \dots ,a_n \in A$, define the morphism 
$\map{\varphi^\bullet}{\cpx{K}{}{\bullet}{A}{M}}{\cpx{C}{}{\bullet}{A}{M}}$ 
by 
	$$
		\varphi^n(f)(a_0 \otimes \cdots \otimes a_{n}) 
		= f(a_0^{(1)} \otimes a_0^{(2)}a_1^{(1)} \otimes \cdots 
		\otimes a_0^{(n+1)}a_1^{(n)} \cdots a_{n-1}^{(2)} a_{n}).
    $$
Then we show that $\varphi^\bullet$ is compatible with the differential. 
Namely, we show 
$\varphi^{n+1} \circ \dif{k}{n} = \dif{C}{n} \circ \varphi^n$. 
For the left hand side, we have
	\begin{align*}
		(&\varphi^{n+1} \circ \dif{K}{n})(f)(a_0 \otimes \cdots \otimes a_{n+1}) \\
		&= \dif{K}{n}(f)(a_0^{(1)} \otimes a_0^{(2)}a_1^{(1)} \otimes 
		\cdots \otimes a_0^{(n+2)}a_1^{(n+1)} \cdots a_{n}^{(2)} a_{n+1}) \\
		&= \sum_{i = 0}^{n+1} (-1)^i f(a_0^{(1)} \otimes \cdots 
		\otimes \varepsilon(a_0^{(i+1)} \cdots a_i^{(1)})a_0^{(i+2)} 
		\cdots a_{i+1}^{(1)} \otimes \cdots \otimes a_0^{(n+2)}a_1^{(n+1)} 
		\cdots a_{n}^{(2)} a_{n+1}) \\
		&= \sum_{i = 0}^{n} (-1)^i f(a_0^{(1)} \otimes \cdots 
		\otimes \varepsilon(a_0^{(i+1)})a_0^{(i+2)}
		\varepsilon(a_1^{(i)})a_1^{(i+1)} \cdots 
		\varepsilon(a_i^{(1)})a_i^{(2)}a_{i+1}^{(1)} \otimes \cdots  \\
		&\quad\, \otimes a_0^{(n+2)}a_1^{(n+1)} \cdots 
		a_{n}^{(2)} a_{n+1}) +(-1)^{n+1} f(a_0^{(1)} \otimes \cdots  
		\otimes a_0^{(n+1)}a_1^{(n)} \cdots a_{n-1}^{(2)} a_{n} 
		\varepsilon(a_{n+1}))\\
		&= \sum_{i = 0}^{n} (-1)^i f(a_0^{(1)} \otimes \cdots 
		\otimes a_0^{(i+1)}a_1^{(i)} \cdots a_{i-1}^{(2)} a_i^{(1)} a_{i+1}^{(1)} 
		\otimes \cdots \otimes a_0^{(n+1)}a_1^{(n)} \cdots  \\
		&\quad\, \times a_i^{(n-i+1)}a_{i+1}^{(n-i+1)} \cdots 
		a_{n}^{(2)} a_{n+1})+(-1)^{n+1} f(a_0^{(1)} \otimes \cdots 
		\otimes a_0^{(n+1)}a_1^{(n)} \cdots a_{n-1}^{(2)} a_{n} \varepsilon(a_{n+1})).
	\end{align*}
	On the other hand, for the right hand side, we have
	\begin{align*}
		(&\dif{C}{n} \circ \varphi^n)(f)(a_0 \otimes \cdots \otimes a_{n+1}) \\
		&= \sum_{i = 0}^{n} (-1)^i \varphi^n(f)(a_0 \otimes \cdots 
		\otimes a_ia_{i+1} \otimes \cdots \otimes a_{n+1}) + (-1)^{n+1}
		\varphi^n(f)(a_0 \otimes \cdots \otimes a_n \varepsilon(a_{n+1})) \\
		&= \sum_{i = 0}^{n} (-1)^i f(a_0^{(1)} \otimes \cdots 
		\otimes a_0^{(i+1)} a_1^{(i)}\cdots a_{i-1}^{(2)} a_i^{(1)}a_{i+1}^{(1)} 
		\otimes \cdots \otimes a_0^{(n+1)}a_1^{(n)} \cdots  \\
		&\quad\, \times a_i^{(n-i+1)}a_{i+1}^{(n-i+1)} \cdots 
		a_{n}^{(2)}a_{n+1})+(-1)^{n+1} f(a_0^{(1)} \otimes \cdots 
		\otimes a_0^{(n+1)}a_1^{(n)} \cdots a_{n-1}^{(2)} a_{n} \varepsilon(a_{n+1})).
	\end{align*}
	Hence, $\varphi^{n+1} \circ \dif{K}{n} = \dif{C}{n} \circ \varphi^n$ holds. 
	While, for any $g \in \cpx{C}{}{n}{A}{M}$ and $a_0 , \dots , a_n \in A$, 
	we define the morphism 
	$\map{\psi^\bullet}{\cpx{C}{}{\bullet}{A}{M}}{\cpx{K}{}{\bullet}{A}{M}}$ by
    $$
		\psi^n (g)(a_0 \otimes \cdots \otimes a_{n}) 
		= g(a_0^{(1)} \otimes S(a_0^{(2)}) a_1^{(1)} 
		\otimes S(a_1^{(2)}) a_2^{(1)} \otimes \cdots 
		\otimes S(a_{n-1}^{(2)}) a_{n}).
    $$
	Then we show that $\psi^\bullet$ is the inverse of $\varphi^\bullet$. 
	For all $g \in \cpx{C}{}{n}{A}{M}$, we have
	\vspace{-0.5em}
	\begin{align*}
		(&\varphi^n \circ \psi^n)(g)(a_0 \otimes \cdots \otimes a_{n}) \\
		 &\quad= \psi^n(g)(a_0^{(1)} \otimes a_0^{(2)} a_1^{(1)} 
		 \otimes \cdots \otimes a_0^{(n+1)} \cdots a_{n-1}^{(2)} a_{n}) \\
		 &\quad= g(a_0^{(1)} \otimes S(a_0^{(2)})a_0^{(3)} a_1^{(1)} 
		 \otimes S(a_0^{(4)}a_1^{(2)})a_0^{(5)}a_1^{(3)}a_2^{(1)} \otimes \cdots \\
		 &\quad\quad\otimes S(a_0^{(2n)}a_1^{(2n-2)} \cdots 
		 a_{n-1}^{(2)})a_0^{(2n+1)}a_1^{(2n-1)}\cdots a_{n-1}^{(3)} a_{n}) \\
		 &\quad= g(a_0^{(1)} \otimes \varepsilon(a_0^{(2)})a_1^{(1)} \otimes 
		 \varepsilon(a_0^{(3)})\varepsilon(a_1^{(2)})a_2^{(1)} \otimes \cdots  
		 \otimes \varepsilon(a_0^{(n+1)})\varepsilon(a_1^{(n)}) \cdots 
		 \varepsilon(a_{n-1}^{(2)}) a_{n}) \\
		 &\quad= g(a_0^{(1)}\varepsilon(a_0^{(2)}) \cdots 
		 \varepsilon(a_0^{(n+1)}) \otimes a_1^{(1)}\varepsilon(a_1^{(2)}) 
		 \cdots \varepsilon(a_1^{(n)}) \otimes \cdots \otimes a_{n-1}^{(1)} 
		 \varepsilon(a_{n-1}^{(2)}) \otimes a_{n}) \\
		 &\quad= g(a_0 \otimes \cdots \otimes a_{n+1}).
	\end{align*}
	Hence, $\varphi^n \circ \psi^n = \id{}$ holds. 
	On the other hand, for all $f \in \cpx{K}{}{n}{A}{M}$, we have
	\begin{align*}
		& (\psi^n \circ \varphi^n)(f)(a_0 \otimes \cdots \otimes a_{n}) \\
		&\quad= \varphi^n(f)(a_0^{(1)} \otimes S(a_0^{(2)}) a_1^{(1)} 
		\otimes \cdots \otimes S(a_{n-1}^{(2)}) a_{n}) \\
		&\quad= f(a_0^{(1)} \otimes a_0^{(2)} S(a_0^{(3)}) a_1^{(1)} 
		\otimes a_0^{(4)} S(a_0^{(5)}) a_1^{(2)} S(a_1^{(3)}) a_2^{(1)} 
		\otimes \cdots \\
		&\quad\quad \otimes 
		a_0^{(2n)}S(a_0^{(2n+1)}) a_1^{(2n-2)}S(a_1^{(2n-1)}) 
		\cdots a_{n-1}^{(2)}S(a_{n-1}^{(3)}) a_{n}) \\
		&\quad= f(a_0^{(1)} \otimes \varepsilon(a_0^{(2)})a_1^{(1)} 
		\otimes \varepsilon(a_0^{(3)}) \varepsilon(a_1^{(2)}) a_2^{(1)} 
		\otimes \cdots \otimes \varepsilon(a_0^{(n+1)}) \cdots 
		\varepsilon(a_{n-1}^{(2)}) a_{n}) \\
		&\quad= f(a_0^{(1)} \varepsilon(a_0^{(2)}) \cdots 
		\varepsilon(a_0^{(n+1)}) \otimes \cdots \otimes a_{n-1}^{(1)} 
		\varepsilon(a_{n-1}^{(2)}) \otimes a_{n}) \\
		&\quad= f(a_0 \otimes \cdots \otimes a_{n}).
	\end{align*}
	Therefore, we have $\psi^n \circ \varphi^n = \id{}$, 
	and hence $\cpx{C}{}{\bullet}{A}{M}$ and $\cpx{K}{}{\bullet}{A}{M}$ 
	are isomorphic.
	
	Next, if $A$ is cocommutative, 
	we will show that $\cpx{CS}{}{\bullet}{A}{M}$ 
	and $\cpx{KS}{}{\bullet}{A}{M}$ are isomorphic 
	by $\varphi^\bullet$ and $\psi^\bullet$. 
	For all $f \in \cpx{KS}{}{n}{A}{M}$ and $1 \leq i \leq n$, we have
	\begin{align*}
		(&\sigma_i \varphi^n (f))(a_0 \otimes \cdots \otimes a_{n}) \\ 
		&\quad= - \varphi^n (f)(a_0 \otimes \cdots \otimes a_{i-1} a_i^{(1)} 
		\otimes S(a_i^{(2)}) \otimes a_i^{(3)} a_{i+1} \otimes \cdots \otimes a_{n}) \\
		&\quad= -f(a_0^{(1)} \otimes a_0^{(2)}a_1^{(1)} \otimes \cdots 
		\otimes a_0^{(i)} \cdots a_{i-2}^{(2)} a_{i-1}^{(1)} a_i^{(1)} 
		\otimes a_0^{(i+1)} \cdots a_{i-1}^{(2)} a_i^{(2)} S(a_i^{(3)}) \\
		&\quad\quad  \otimes a_0^{(i+2)} 
		\cdots a_{i-1}^{(3)} a_i^{(4)} a_{i+1}^{(1)}\otimes \cdots 
		\otimes a_0^{(n+1)}\cdots a_{n-1}^{(2)}a_{n}) \\
		&\quad= -f(a_0^{(1)} \otimes a_0^{(2)}a_1^{(1)} \otimes \cdots 
		\otimes a_0^{(i)} \cdots a_{i-2}^{(2)} a_{i-1}^{(1)} a_i^{(1)} 
		\otimes a_0^{(i+1)} \cdots a_{i-1}^{(2)} \varepsilon(a_i^{(2)}) \\
		&\quad\quad  \otimes a_0^{(i+2)} 
		\cdots a_{i-1}^{(3)} a_i^{(3)}a_{i+1}^{(1)}\otimes \cdots 
		\otimes a_0^{(n+1)}\cdots a_{n-1}^{(2)}a_{n}) \\
		&\quad= -f(a_0^{(1)} \otimes a_0^{(2)}a_1^{(1)} \otimes 
		\cdots \otimes a_0^{(i)} \cdots a_{i-2}^{(2)} a_{i-1}^{(1)} a_i^{(1)} 
		\otimes a_0^{(i+1)} \cdots a_{i-1}^{(2)}\otimes a_0^{(i+2)} 
		\cdots a_{i-1}^{(3)} a_i^{(2)}a_{i+1}^{(1)} \\
		&\quad\quad  \otimes \cdots \otimes a_0^{(n+1)}
		\cdots a_{n-1}^{(2)}a_{n}) \\
		&\quad= -(\sigma_i f)(a_0^{(1)} \otimes a_0^{(2)}a_1^{(1)} 
		\otimes \cdots \otimes a_0^{(i)} \cdots 
		a_{i-2}^{(2)} a_{i-1}^{(1)} a_i^{(1)} \otimes a_0^{(i+1)} 
		\cdots a_{i-1}^{(2)} \\
		&\quad\quad \otimes a_0^{(i+2)} \cdots 
		a_{i-1}^{(3)} a_i^{(2)} a_{i+1}^{(1)}\otimes \cdots \otimes a_0^{(n+1)}
		\cdots a_{n-1}^{(2)}a_{n}) \\
		&\quad= f (a_0^{(1)} \otimes a_0^{(2)}a_1^{(1)} \otimes 
		\cdots \otimes a_0^{(i)} \cdots a_{i-1}^{(1)} \otimes 
		a_0^{(i+1)} \cdots a_{i-2}^{(3)} a_{i-1}^{(2)} a_i^{(1)}
		\otimes a_0^{(i+2)} \cdots a_{i-1}^{(3)} a_i^{(2)} a_{i+1}^{(1)} \\
		&\quad\quad \otimes \cdots \otimes a_0^{(n+1)}
		\cdots a_{n-1}^{(2)}a_{n}) \\
		&\quad= \varphi^n (f)(a_0 \otimes \cdots \otimes a_{n}).
	\end{align*}
	Hence, $\varphi^n(f) \in \cpx{CS}{}{n}{A}{M}$ holds. 
	While, for all $g \in \cpx{CS}{}{n}{A}{M}$ and $1 \leq i \leq n$, 
	we have 
	\begin{align*}
		&(\sigma_i \psi^n (g))(a_0 \otimes \cdots \otimes a_{n}) \\
		&\quad= - \psi^n (g)(a_0 \otimes \cdots \otimes a_i 
		\otimes a_{i-1} \otimes \cdots \otimes a_{n}) \\
		&\quad= -g(a_0^{(1)} \otimes S(a_0^{(2)})a_1^{(1)} 
		\otimes \cdots \otimes S(a_{i-2}^{(2)}) a_i^{(1)} 
		\otimes S(a_i^{(2)}) a_{i-1}^{(1)} \otimes S(a_{i-1}^{(2)})a_{i+1}^{(1)}\\ 
		&\quad\quad \otimes \cdots \otimes S(a_{n-1}^{(2)})a_{n}) \\
		&\quad= -(\sigma_i g)(a_0^{(1)} \otimes S(a_0^{(2)})a_1^{(1)} 
		\otimes \cdots \otimes S(a_{i-2}^{(2)}) a_i^{(1)} \otimes 
		S(a_i^{(2)}) a_{i-1}^{(1)} \otimes S(a_{i-1}^{(2)})a_{i+1}^{(1)}\\
		&\quad\quad \otimes \cdots \otimes S(a_{n-1}^{(2)})a_{n}) \\
		&\quad= g(a_0^{(1)} \otimes S(a_0^{(2)})a_1^{(1)} \otimes 
		\cdots \otimes S(a_{i-2}^{(2)}) a_i^{(1)} S(a_i^{(2)})a_{i-1}^{(1)} 
		\otimes S(S(a_i^{(3)}) a_{i-1}^{(2)}) \\
		&\quad\quad  \otimes S(a_i^{(4)}) a_{i-1}^{(3)} S(a_{i-1}^{(4)}) 
		a_{i+1}^{(1)}\otimes \cdots \otimes S(a_{n-1}^{(2)})a_{n}) \\
		&\quad= g(a_0^{(1)} \otimes S(a_0^{(2)})a_1^{(1)} 
		\otimes \cdots \otimes S(a_{i-2}^{(2)}) 
		\varepsilon(a_i^{(1)}) a_{i-1}^{(1)} \otimes S( a_{i-1}^{(2)})a_i^{(2)}
		\otimes S(a_i^{(3)}) \varepsilon(a_{i-1}^{(3)}) a_{i+1}^{(1)} \\
		&\quad\quad  \otimes \cdots \otimes S(a_{n-1}^{(2)})a_{n}) \\
		&\quad= g(a_0^{(1)} \otimes S(a_0^{(2)})a_1^{(1)} \otimes 
		\cdots \otimes S(a_{i-2}^{(2)}) a_{i-1}^{(1)} 
		\otimes S( a_{i-1}^{(2)})a_i^{(1)} \otimes S(a_i^{(2)}) a_{i+1}^{(1)} 
		\otimes \cdots \otimes S(a_{n-1}^{(2)})a_{n}) \\
		&\quad= \psi^n (g)(a_0 \otimes \cdots \otimes a_{n}).
	\end{align*} 
	So, we have $\psi^n(g) \in \cpx{KS}{}{n}{A}{M}$. 
	Therefore, $\cpx{CS}{}{\bullet}{A}{M}$ and $\cpx{KS}{}{\bullet}{A}{M}$ 
	are isomorphic as complexes. 
	Moreover, $\shcoho{n}{A}{M} = \coho{n}{\cpx{KS}{}{\bullet}{A}{M}}$ holds.
\end{proof}


\subsection{Definition of symmetric Hochschild cohomology}
In this subsection, we recall the definition of Hochschild cohomology 
and define symmetric Hochschild cohomology 
for cocommutative Hopf algebras. 
\begin{defin}[cf. {\cite[Section 1.1]{witherspoon2019hochschild}}]
	Let $A$ be a Hopf algebra and $M$ an $A$-bimodule. 
	The \textit{Hochschild cohomology} $\hoch{\bullet}{A}{M}$ of $A$ 
	with coefficients in $M$ is defined by 
    $$
		\hoch{n}{A}{M} = \ext{A^e}{n}{A}{M}.
    $$
\end{defin}

We construct a standard non-homogeneous complex 
which gives the Hochschild cohomology. 
Let $n$ be an integer such that $n \geq 0$, 
and let $\res{T}{n}{e}{A} = A^{\otimes n+2}$ be an $A$-bimodule via
$$
	(b \otimes c^\op) \cdot (a_0 \otimes a_1 \otimes \cdots \otimes a_{n+1}) 
	= ba_0 \otimes a_1 \otimes \cdots \otimes a_{n+1}c
$$
for any $b,c, a_0, \dots , a_n \in A$. 
Then there is a projective resolution of $A$ as $A^\e$-modules:
$$
	\cdots \longrightarrow \res{T}{n}{e}{A} 
	\xrightarrow{\diff{n}{T^\e}} \res{T}{n-1}{e}{A} 
	\longrightarrow 
	\cdots \longrightarrow 
	\res{T}{1}{e}{A} \xrightarrow{\diff{1}{T^\e}} 
	\res{T}{0}{e}{A} \xrightarrow{\diff{0}{T^\e}} A \longrightarrow 0,
$$
where we set
\begin{align*}
	&\diff{0}{T^\e}(a_0 \otimes a_1) = a_0a_1 , \\
	&\diff{n}{T^\e}(a_0 \otimes a_1 \otimes \cdots \otimes a_{n+1}) 
	= \sum_{i=0}^{n} (-1)^i a_0 \otimes a_1 \otimes \cdots 
	\otimes a_i a_{i+1} \otimes \cdots \otimes a_{n+1}, \text{ for } n \geq 1.
\end{align*}
Moreover, we denote $\cpx{C}{e}{n}{A}{M} = \Hom{A}{\res{T}{n}{e}{A}}{M}$ 
and $\dif{C_e}{n} = \Hom{A}{\diff{n+1}{T^e}}{M}$. 
Next, we define an action of the symmetric group $S_{n+1}$ 
on $\cpx{C}{e}{n}{A}{M}$ for each $n \geq 0$. 
Let $A$ be a cocommutative Hopf algebra. 
Then we define the action of $\sigma_i = (i,i+1) \in S_{n+1}$ 
on $\cpx{C}{e}{n}{A}{M}$ by 
\begin{align}\label{eq:sa3}
	(\sigma_i f)(a_0 \otimes \cdots \otimes a_{n+1}) 
	= - f(a_0 \otimes \cdots \otimes a_{i-1}a_i^{(1)} \otimes S(a_i^{(2)}) 
	\otimes a_i^{(3)} a_{i+1} \otimes \cdots \otimes a_{n+1})
\end{align}
for $f \in \cpx{C}{e}{n}{A}{M}$, $a_0 , \dots , a_n \in A$ 
and $1 \leq i \leq n$. 
Similar to the case of symmetric cohomology, 
the above formula \eqref{eq:sa3} is well-defined as an action, 
and the action is compatible with the differential. 
Hence the sequence of invariant subspaces 
$\cpx{C}{e}{\bullet}{A}{M}^{S_{\bullet +1}}$ 
by the action of symmetric groups 
becomes a subcomplex of $\cpx{C}{e}{\bullet}{A}{M}$ 
and is denoted by $\cpx{CS}{e}{\bullet}{A}{M}$.


\begin{defin}
	Let $A$ be a cocommutative Hopf algebra and $M$ an $A$-bimodule. 
	The \textit{symmetric Hochschild cohomology} $\shoch{\bullet}{A}{M}$ of $A$ 
	with coefficients in $M$ is defined by 
	\begin{align*}
		\shoch{n}{A}{M} = \coho{n}{\cpx{CS}{e}{\bullet}{A}{M}}.
	\end{align*}
\end{defin}
Similar to Hopf algebra cohomology, 
we can construct a homogeneous complex 
which gives the Hochschild cohomology. 
We describe its construction and an action of symmetric groups 
on the homogeneous complex. 
Let $A$ be a Hopf algebra and $n$ a non-negative integer. 
Suppose that $\res{\widetilde{T}}{n}{e}{A} = A^{\otimes n+2}$ 
is an $A$-bimodule via
\begin{align*}
	 (b \otimes c^\op) \cdot (a_0 \otimes a_1 \otimes \cdots 
	 \otimes a_{n+1}) = b^{(1)} a_0 \otimes b^{(2)} a_1 \otimes \cdots 
	 \otimes b^{(n+2)} a_{n+1} c
\end{align*}
for any $b,c,a_1, \dots ,a_{n+1} \in A$. 
Then there is a projective resolution of $A$ as $A^\e$-modules:
\begin{align*}
	\cdots \longrightarrow \res{\widetilde{T}}{n}{e}{A} 
	\xrightarrow{\diff{n}{\widetilde{T}^\e}} \res{\widetilde{T}}{n-1}{e}{A} 
	\longrightarrow 
	\cdots \longrightarrow 
	\res{\widetilde{T}}{1}{e}{A} \xrightarrow{\diff{1}{\widetilde{T}^\e}} 
	\res{\widetilde{T}}{0}{e}{A} \xrightarrow{\diff{0}{\widetilde{T}^\e}} 
	A \longrightarrow 0,
\end{align*}
where we set 
$\diff{n}{\widetilde{T}^\e}(a_0 \otimes a_1 \otimes \cdots \otimes a_{n+1})
= \sum_{i=0}^{n} (-1)^i a_0 \otimes a_1 \otimes \cdots \otimes 
\varepsilon(a_i) a_{i+1} \otimes \cdots \otimes a_{n+1}$. 
Moreover, we denote 
$\cpx{K}{e}{n}{A}{M} = \Hom{A}{\res{\widetilde{T}}{n}{e}{A}}{M}$ 
and $\dif{K_e}{n} = \Hom{A}{\diff{n+1}{\widetilde{T^e}}}{M}$. 
Next, we define an action of $S_{n+1}$ on $\cpx{K}{e}{n}{A}{M}$ 
for each $n \geq 0$. 
Let $A$ be a cocommutative Hopf algebra. 
Then we define the action of $\sigma_i$ on $\cpx{K}{e}{n}{A}{M}$ by 
\begin{align}\label{eq:sa4}
	(\sigma_i f) (a_0 \otimes \dots \otimes a_{n+1}) 
	= -f(a_0 \otimes \dots \otimes a_i \otimes a_{i-1} \otimes \dots 
	\otimes a_{n+1}), \text{ for } 1\leq i \leq n
\end{align}
for $f \in \cpx{K}{e}{n}{A}{M}$ and $a_0, \dots , a_{n+1} \in A$. 
Similar to the case of symmetric cohomology, 
the above formula \eqref{eq:sa4} is well-defined as an action, 
and the action is compatible with the differential. 
Hence the sequence of invariant subspaces 
$\cpx{K}{e}{\bullet}{A}{M}^{S_{\bullet +1}}$ 
by the action of symmetric groups 
becomes a subcomplex of $\cpx{K}{e}{\bullet}{A}{M}$ 
and is denoted by $\cpx{KS}{e}{\bullet}{A}{M}$. 
Moreover, we have the following assertion which is similar to 
Proposition \ref{prop.nonhomo-homo-iso}. 

\begin{prop}
\label{prop.nonhomo-homo-iso-hoch}
Let $A$ be a Hopf algebra and $M$ an $A$-bimodule. 
Then there is an isomorphism $\cpx{C}{e}{\bullet}{A}{M} \cong \cpx{K}{e}{\bullet}{A}{M}$ 
as complexes. Moreover, if $A$ is cocommutative, then this isomorphism induces an isomorphism 
$\cpx{CS}{e}{\bullet}{A}{M} \cong \cpx{KS}{e}{\bullet}{A}{M}$ as complexes. 
Therefore, $\shoch{n}{A}{M} = \coho{n}{\cpx{KS}{e}{\bullet}{A}{M}}$ holds.
\end{prop}
\section{The relationships between classical, symmetric and symmetric Hochschild cohomology}
\subsection{Resolutions that give symmetric cohomology and symmetric Hochschild cohomology}
In this subsection, we describe that symmetric cohomology 
is given by a resolution of $k$. 
For $a_0 , \dots , a_n \in A$ and $1 \leq i \leq n$, 
define the right action of $S_{n+1}$ on $\res{\widetilde{T}}{n}{}{A}$ by 
$$
	(a_0 \otimes \cdots \otimes a_n) \cdot \sigma_i 
	= -a_0 \otimes \cdots \otimes a_{i-2} \otimes a_i \otimes a_{i-1} 
	\otimes a_{i+1} \otimes \cdots \otimes a_n.
$$
This action implies that $\res{\widetilde{T}}{n}{}{A}$ 
is a right $kS_{n+1}$-module. 
Furthermore, $\cpx{K}{}{n}{A}{M}$ is a left $kS_{n+1}$-module 
induced by the right $kS_{n+1}$-module structure 
of $\res{\widetilde{T}}{n}{}{A}$. 
According to the structure of the counit on $kS_{n+1}$, 
there is a natural isomorphism of $k$-vector spaces 
$\cpx{KS}{}{n}{A}{M} \cong {}^{kS_{n+1}}\cpx{K}{}{n}{A}{M}$. 
Moreover, we show that it can be deformed $\cpx{KS}{}{\bullet}{A}{M}$ 
to a form of the set of homomorphisms. 
\begin{prop} \label{prop.ks-hom}
	Let $A$ be a cocommutative Hopf algebra and $M$ a left $A$-module. 
	Then, for each $n \geq 0$, there is an isomorphism
$$
		\cpx{KS}{}{n}{A}{M} \cong \Hom{A}{\res{\widetilde{T}}{n}{}{A} 
		\otimes_{kS_{n+1}} k}{M}
$$
	as $k$-vector spaces with the trivial left $kS_{n+1}$-module $k$.
\end{prop}

\begin{proof}
	According to Lemma \ref{lem.hopf-hom-invariant} 
	and the adjunction isomorphism of the tensor product 
	and the set of homomorphisms, we have
	\begin{align*}
		\Hom{A}{\res{\widetilde{T}}{n}{}{A}}{M}^{S_{n+1}} 
		&\cong {}^{kS_{n+1}}\Hom{A}{\res{\widetilde{T}}{n}{}{A}}{M} 
		\cong {}^{kS_{n+1}}\Hom{k}{k}{\Hom{A}{\res{\widetilde{T}}{n}{}{A}}{M}} \\
		&\cong \Hom{kS_{n+1}}{k}{\Hom{A}{\res{\widetilde{T}}{n}{}{A}}{M}} 
		\cong \Hom{A}{\res{\widetilde{T}}{n}{}{A} \otimes_{kS_{n+1}} k}{M}.
	\end{align*}
\end{proof}

Next, we define a resolution of $k$. 
We take the left $A$-module sequence:
$$
	\cdots \longrightarrow  \sym{n}{}{A} 
	\overset{\diff{n}{\widetilde{S}}}{\longrightarrow}  
	\sym{n-1}{}{A} \longrightarrow 
	\cdots \longrightarrow  \sym{1}{}{A} 
	\overset{\diff{1}{\widetilde{S}}}{\longrightarrow} 
	\sym{0}{}{A} \overset{\diff{0}{\widetilde{S}}}{\longrightarrow} 
	k \longrightarrow 0,
$$
where we set
$\diff{n}{\widetilde{S}}((a_0 \otimes \cdots \otimes a_n) \otimes_{kS_{n+1}} x) 
= \diff{n}{\widetilde{T}}(a_0 \otimes \cdots \otimes a_n) \otimes_{kS_{n}} x$
and denote	$\sym{n}{}{A} = \res{\widetilde{T}}{n}{}{A} \otimes_{kS_{n+1}}k$. 
According to the relation between the action and the differential, 
the morphism 
$$
\map{\diff{n}{}}{\res{\widetilde{T}}{n}{}{A} \times k}{\sym{n-1}{}{A}}; 
\ ((a_0 \otimes \cdots \otimes a_n) , x) 
\mapsto \diff{n}{\widetilde{T}}(a_0 \otimes \cdots \otimes a_n) 
\otimes_{kS_n} x
$$
satisfies the relation 
$\diff{n}{}((a_0 \otimes \cdots \otimes a_n) \cdot \sigma_i , x) 
= \diff{n}{}((a_0 \otimes \cdots \otimes a_n) , \sigma_i \cdot x)$ 
for each $1 \leq i \leq n$. Hence, $\diff{n}{\tilde{S}}$ is well-defined. 
Moreover, we define the morphism 
$\map{h_n^{\mathrm{\tilde{S}}}}{\sym{n}{}{A}}{\sym{n+1}{}{A}}$ by 
$h_n^{\mathrm{\tilde{S}}}((a_0 \otimes \cdots \otimes a_n) 
\otimes_{kS_{n+1}} x) = (1 \otimes a_0 \otimes \cdots \otimes a_n) 
\otimes_{kS_{n+2}} x$, 
then this morphism is a contracting homotopy. 
Therefore, $(\sym{\bullet}{}{A} , \diff{\bullet}{\tilde{S}})$ is exact. 

Similarly, we describe that symmetric Hochschid cohomology 
is given by a resolution of $A$. 
For $a_0 , \dots , a_{n+1} \in A$ and $1 \leq i \leq n$, 
define the right action of $S_{n+1}$ on $\res{\widetilde{T}}{n}{e}{A}$ by 
$$
	(a_0 \otimes \cdots \otimes a_{n+1}) \cdot \sigma_i 
	= -a_0 \otimes \cdots \otimes a_{i-2} \otimes a_i \otimes a_{i-1} 
	\otimes a_{i+1} \otimes \cdots \otimes a_{n+1}.
$$
This action implies that $\res{\widetilde{T}}{n}{e}{A}$ 
is a right $kS_{n+1}$-module. 
Furthermore, $\cpx{K}{e}{n}{A}{M}$ is a left $kS_{n+1}$-module 
induced by the right $kS_{n+1}$-module structure 
of $\res{\widetilde{T}}{n}{e}{A}$. 
According to the structure of the counit on $kS_{n+1}$, 
there is a natural isomorphism 
$\cpx{KS}{e}{n}{A}{M} \cong {}^{kS_{n+1}} \cpx{K}{e}{n}{A}{M}$ 
as $k$-vector spaces. 
Moreover, we get an isomorphism which is similar to 
Proposition \ref{prop.ks-hom}.
\begin{prop}
	Let $A$ be a cocommutative Hopf algebra and $M$ an $A$-bimodule. 
	Then, for each $n \geq 0$,  there is an isomorphism
    $$
		\cpx{KS}{e}{n}{A}{M} \cong 
		\Hom{A}{\res{\widetilde{T}}{n}{e}{A} \otimes_{kS_{n+1}} k}{M}
	$$
	as $k$-vector spaces, with the trivial left $kS_{n+1}$-module $k$.
\end{prop}

Next, we define a resolution of $A$. 
We take the $A$-bimodule sequence:
$$
	\cdots \longrightarrow  \sym{n}{e}{A} 
	\overset{\diff{n}{\widetilde{S}^\e}}{\longrightarrow}  
	\sym{n-1}{e}{A} \longrightarrow 
	\cdots \longrightarrow  \sym{1}{e}{A} 
	\overset{\diff{1}{\widetilde{S}^\e}}{\longrightarrow} 
	\sym{0}{e}{A} \overset{\diff{0}{\widetilde{S}^\e}}{\longrightarrow} 
	A \longrightarrow 0,
$$
where we set 
$\diff{n}{\widetilde{S}^\e}((a_0 \otimes \cdots \otimes a_{n+1}) 
\otimes_{kS_{n+1}} x) 
= \diff{n}{\widetilde{T}^\e}(a_0 \otimes \cdots \otimes a_{n+1}) 
\otimes_{kS_{n}} x$ 
and denote $\sym{n}{}{A} = \res{\widetilde{T}}{n}{}{A} \otimes_{kS_{n+1}}k$. 
Then this sequence is a resolution of $A$ as $A^\e$-modules. 
Moreover, there is an isomorphism 
$\sym{n}{e}{A} \cong \sym{n}{}{A} \otimes A$ 
as $A$-bimodules where the right $A$-module structure 
of $\sym{n}{}{A}$ is trivial.
\subsection{The relationships between symmetric cohomology and symmetric Hochschild cohomology}
In this subsection, 
we show that there is an isomorphism 
between symmetric cohomology and symmetric Hochschild cohomology 
as $k$-vector spaces. 
First, we describe the fact about group cohomology 
by Eilenberg and MacLane. 
Here, we recall that ${}^\ad X$ is a left module 
whose structure given by the left adjoint action. 
According to the structure of the coproduct and the antipode 
of group algebras, ${}^\ad X$ is a left $G$-module 
by $g \cdot x = g x g^{-1}$ for $g \in G$ and $x \in X$. 
\begin{thm}[Eilenberg-MacLane {\cite[Section 5]{eilenberg1947cohomology}}]
\label{thm.iso-grp-coho}
Let $G$ be a group and X a $G$-bimodule. 
Then, for each $n \geq 0$, there is an isomorphism 
	$\hoch{n}{\Z G}{X} \cong \hcoho{n}{G}{{}^\ad X}$
as $\Z$-modules.
\end{thm}

\begin{rem}
\label{rem.iso-hopf-coho}
Theorem \ref{thm.iso-grp-coho} is generalized 
to the case of Hopf algebras 
by Ginzburg-Kumar \cite[Subsection 5.6]{ginzburg1993cohomology}.	 
\end{rem}

In this paper, we have the following result 
which is a symmetric version of the above isomorphism. 

\begin{thm}
\label{thm.iso-symcoho-symHoch}
	If $A$ be a cocommutative Hopf algebra and $M$ an $A$-bimodule. 
	Then, for each $n \geq 0$, there is an isomorphism
    $$
		\shoch{n}{A}{M} \cong \shcoho{n}{A}{{}^\ad M}
    $$
	as $k$-vector spaces.
\end{thm}
\begin{proof}
	According to Lemma \ref{lem.hopf-adjunction}, 
	we have the following isomorphisms
	\begin{align*}
		\cpx{KS}{e}{n}{A}{M} &\cong \Hom{A^e}{\sym{n}{e}{A}}{M} \\
								   &= \Hom{A^e}{\sym{n}{}{A} \otimes A}{M} 
								   \cong \Hom{A^e}{\sym{n}{}{A}}{\Hom{k}{A}{M}}.
	\end{align*}
	So, since $\sym{n}{}{A}$ is a trivial right $A$-module, 
	there is an isomorphism 
    $$
		 \Hom{A^e}{\sym{n}{}{A}}{\Hom{k}{A}{M}} \cong 
		 \Hom{l\operatorname{\mathsf{\hspace{-1pt}-\hspace{-1pt}}}A}{\sym{n}{}{A}}{\Hom{k}{A}{M}^A},
	$$
	where $\mathrm{Hom}_{l\operatorname{\mathsf{\hspace{-1pt}-\hspace{-1pt}}}A}$ (or $\mathrm{Hom}_{r\operatorname{\mathsf{\hspace{-1pt}-\hspace{-1pt}}}A}$) denotes a set of homomorphisms as a left (or right) $A$-module. While, there are isomorphisms of left $A$-modules
    $$
		\Hom{k}{A}{M}^A \cong 
		\Hom{r\operatorname{\mathsf{\hspace{-1pt}-\hspace{-1pt}}}A}{A}{M} 
		\cong {}^{\ad}{M}.
    $$
	Hence, we have 
	$\Hom{l\operatorname{\mathsf{\hspace{-1pt}-\hspace{-1pt}}}A}{\sym{n}{}{A}}{\Hom{k}{A}{M}^A} 
	\cong \Hom{l\operatorname{\mathsf{\hspace{-1pt}-\hspace{-1pt}}}A}{\sym{n}{}{A}}{{}^{\ad}M} 
	\cong \cpx{KS}{}{n}{A}{{}^{\ad}M}$.
\end{proof}
Moreover, we have the following assertion 
from Theorem \ref{thm.iso-symcoho-symHoch}.
\begin{cor}
	If $A$ be a finite dimensional, commutative 
	and cocommutative Hopf algebra. 
	Then, for each $n \geq 0$, there is an isomorphism 
$$
		\shoch{n}{A}{A} \cong A \otimes \shcoho{n}{A}{k}
$$
	as $k$-vector spaces.
\end{cor}
\begin{proof}
	According to Theorem \ref{thm.iso-symcoho-symHoch}, 
	there is an isomorphism 
	$\cpx{KS}{e}{n}{A}{A} \cong \cpx{KS}{}{n}{A}{{}^{\ad}{A}}$.
	Furthermore, by Lemma \ref{lem.hopf-hom-invariant} 
	and \ref{lem.hopf-hom-dual}, there are isomorphisms
	\begin{align*}
		\cpx{KS}{}{n}{A}{{}^{\ad}A} 
		&\cong \Hom{A}{\sym{n}{}{A}}{{}^{\ad}A} \\
		&\cong {}^{A}\Hom{k}{\sym{n}{}{A}}{{}^{\ad}A} 
		\cong {}^{A}({}^{\ad}A \otimes \Hom{k}{\sym{n}{}{A}}{k}).
	\end{align*}
	Since $A$ is commutative, we have
	\begin{align*}
		{}^{A}({}^{\ad}A \otimes \Hom{k}{\sym{n}{}{A}}{k}) 
		&\cong A \otimes {}^{A}\Hom{k}{\sym{n}{}{A}}{k} \\
		&\cong A \otimes \Hom{A}{\sym{n}{}{A}}{k}
	    \cong A \otimes \cpx{KS}{}{n}{A}{k}.
	\end{align*}
\end{proof}
\subsection{The relationships between classical cohomology and symmetric cohomology}
In this subsection, we describe the relationships 
between classical cohomology and symmetric cohomology. 
First, the results about the morphism 
$\shcoho{n}{A}{M} \longrightarrow \hcoho{n}{A}{M}$ 
induced by the inclusion in low degrees are as follows. 
In the case of degree $0$, since $S_1$ is the trivial group, 
we have $\cpx{CS}{}{0}{A}{M} = \Hom{A}{\res{\widetilde{T}}{0}{}{A}}{M}$. 
In the case of degree $1$, we have the following assertion 
by the same proofs as for \cite[Proposition 2.1]{todea2015symmetric}.
\begin{prop}
	Let $A$ be a cocommutative Hopf algebra 
	and $M$ a left $A$-module. 
	Then there is an isomorphism 
	$\shcoho{1}{A}{M} \cong \hcoho{1}{A}{M}$ as $k$-vector spaces.
\end{prop}
In the case of degree $2$, 
we have the following assertion by the same proofs 
as for \cite[Lemma 3.1]{staic2009symmetric}. 
\begin{prop}
	Let $A$ be a cocommutative Hopf algebra 
	and $M$ a left $A$-module. 
	Then the morphism 
	$\shcoho{2}{A}{M} \longrightarrow \hcoho{2}{A}{M}$ 
	induced by the inclusion 
	$\cpx{CS}{}{2}{A}{M} \longrightarrow \cpx{C}{}{2}{A}{M}$ is injective.
\end{prop}
Next, we have the following result about the projectivity 
of the resolution $\sym{\bullet}{}{A}$.
\begin{thm}\label{thm.res-proj}
	Let $A$ be a cocommutative Hopf algebra. 
	If $\ch{k} \nmid n+1$, then $\sym{n}{}{A}$ is a projective $A$-module 
	for each $n \geq 1$.
\end{thm}
\begin{proof}
	First, we show that $\sym{n}{}{A}$ is a direct summand 
	of $A \otimes \sym{n-1}{}{A}$ as a left $A$-module. 
	Define the morphism 
	$\map{\phi}{\sym{n}{}{A}}{A \otimes \sym{n-1}{}{A}}$ by
    $$
		\phi( (a_0 \otimes \cdots \otimes a_n) \otimes_{kS_{n+1}} x) 
		= \frac{1}{n+1} \sum_{i = 0}^n (-1)^i a_i \otimes
		 ((a_0 \otimes \cdots \otimes a_{i-1} \otimes a_{i+1} 
		 \otimes \cdots \otimes a_n) \otimes_{kS_n} x)
    $$
	for any $a_0, \dots , a_n \in A$ and $x \in k$. 
	Then this morphism is well-defined. 
	Indeed, we define the morphism 
	$\map{\phi^\prime}{\res{\widetilde{T}}{n}{}{A} \times k}{A \otimes \sym{n-1}{}{A}}$ 
	by
    $$
		\phi^\prime( (a_0 \otimes \cdots \otimes a_n) , x) 
		= \frac{1}{n+1} \sum_{i = 0}^n (-1)^i a_i 
		\otimes ((a_0 \otimes \cdots \otimes a_{i-1} \otimes a_{i+1}
		\otimes \cdots \otimes a_n) \otimes_{kS_n} x).
    $$
	We show that 
	$\phi^\prime ((a_0 \otimes \cdots \otimes a_n)\cdot \sigma_j , x) 
	= \phi^\prime((a_0 \otimes \cdots \otimes a_n) , \sigma_j \cdot x)$ 
	for each $1 \leq j \leq n$. 
	The left hand side is 
	\begin{align*}
		&\phi^\prime ((a_0 \otimes \cdots \otimes a_n)\cdot \sigma_j , x) 
		= \phi^\prime ((-a_0 \otimes \cdots \otimes a_j 
		\otimes a_{j-1} \otimes \cdots \otimes a_n) , x) \\
		&= \frac{1}{n+1} \left\{ \sum_{i=0}^{j-2} (-1)^i a_i 
		\otimes ((-a_0 \otimes \cdots \otimes a_{i-1} \otimes a_{i+1}
		\otimes \cdots \otimes a_j \otimes a_{j-1} \otimes \cdots 
		\otimes a_n) \otimes_{kS_{n}} x)\right. \\
		&\hspace{1.3cm} +(-1)^{j-1} a_j \otimes 
		((-a_0 \otimes \cdots \otimes a_{j-2} \otimes 
		a_{j-1} \otimes a_{j+1} \otimes \cdots \otimes a_n) 
		\otimes_{kS_{n}} x) \\
		&\hspace{1.3cm} +(-1)^j a_{j-1} 
		\otimes ((-a_0 \otimes \cdots \otimes a_{j-2} 
		\otimes a_j \otimes a_{j+1} \otimes \cdots \otimes a_n) 
		\otimes_{kS_{n}} x) \\
		&\hspace{1.3cm} +\left.\sum_{i=j+1}^n (-1)^i a_i 
		\otimes ((-a_0 \otimes \cdots \otimes a_j \otimes a_{j-1} 
		\otimes \cdots \otimes a_{i-1} \otimes a_{i+1} \otimes \cdots \otimes a_n) 
		\otimes_{kS_{n}} x) \right\}.
	\end{align*}
	By deforming the first term of the above formula, 
	because $k$ is a trivial left $kS_n$-module, we have
\begin{align*}
	&(-a_0 \otimes \cdots \otimes a_{j-1} \otimes a_{j+1}
	\otimes \cdots 
	\otimes a_j \otimes a_{j-1} \otimes \cdots \otimes a_n) 
	\otimes_{kS_{n}} x \\
	&\quad=(-a_0 \otimes \cdots \otimes a_{j-1} \otimes a_{j+1}
	\otimes \cdots \otimes a_j \otimes a_{j-1} 
	\otimes \cdots \otimes a_n) \otimes_{kS_{n}} \sigma_{j-1} \cdot x \\
	&\quad=(-a_0 \otimes \cdots \otimes a_{j-1} \otimes a_{j+1}
	\otimes \cdots \otimes a_j \otimes a_{j-1} \otimes \cdots 
	\otimes a_n)\cdot \sigma_{j-1} \otimes_{kS_{n}} x \\
	&\quad=(a_0 \otimes \cdots \otimes a_{j-1} \otimes a_{j+1} \otimes 
	\cdots \otimes a_{j-1} \otimes a_{j} \otimes \cdots \otimes a_n) 
	\otimes_{kS_{n}} x.
\end{align*}
 Similarly, we have
\begin{align*}
	&(-a_0 \otimes \cdots \otimes a_j \otimes a_{j-1} \otimes 
	\cdots \otimes a_{i-1} \otimes a_{i+1} \otimes \cdots \otimes a_n) 
	\otimes_{kS_{n}} x \\
	&\quad=(a_0 \otimes \cdots \otimes a_{j-1} \otimes a_{j} \otimes 
	\cdots \otimes a_{i-1} \otimes a_{i+1} \otimes \cdots \otimes a_n) 
	\otimes_{kS_{n}} x,
\end{align*}
	and so we have
	$$
	\phi^\prime ((a_0 \otimes \cdots \otimes a_n)\cdot \sigma_j , x) 
	= \phi^\prime ((a_0 \otimes \cdots \otimes a_n) , x) 
	=\phi^\prime ((a_0 \otimes \cdots \otimes a_n) , \sigma_j \cdot x).
	$$
	Therefore, $\phi$ is well-defined. 
	Next, we show that $\phi$ is an $A$-module homomorphism. 
	A direct calculation shows that, for all $a, a_0, \dots , a_n \in A$ 
	and $x \in k$, 
	\vspace{-0.5em}
\begin{align*}
	&\phi( a\cdot ((a_0 \otimes \cdots \otimes a_n) 
	\otimes_{kS_{n+1}} x)) \\
	&\quad= \phi( a\cdot (a_0 \otimes \cdots \otimes a_n) 
	\otimes_{kS_{n+1}} x) \\
	&\quad= \phi( (a^{(1)}a_0 \otimes \cdots \otimes a^{(n+1)}a_n) 
	\otimes_{kS_{n+1}} x) \\
	&\quad= \frac{1}{n+1} \sum_{i=0}^n (-1)^i a^{(i+1)}a_i 
	\otimes ((a^{(1)}a_0 \otimes \cdots \otimes a^{(i)}a_{i-1} 
	\otimes a^{(i+2)}a_{i+1} \otimes \cdots \otimes a^{(n+1)}a_n) 
	\otimes_{kS_n} x) \\
	&\quad= \frac{1}{n+1} \sum_{i=0}^n (-1)^i a^{(1)}a_i 
	\otimes ((a^{(2)}a_0 \otimes \cdots \otimes a^{(i+1)}a_{i-1} 
	\otimes a^{(i+2)}a_{i+1} \otimes \cdots \otimes a^{(n+1)}a_n) 
	\otimes_{kS_n} x) \\
	&\quad= \frac{1}{n+1} \sum_{i=0}^n (-1)^i a^{(1)}a_i 
	\otimes (a^{(2)}\cdot (a_0 \otimes \cdots \otimes a_{i-1} 
	\otimes a_{i+1} \otimes \cdots \otimes a_n) \otimes_{kS_n} x) \\
	&\quad= \frac{1}{n+1} \sum_{i=0}^n (-1)^i a\cdot 
	(a_i \otimes ((a_0 \otimes \cdots \otimes a_{i-1} \otimes a_{i+1} 
	\otimes \cdots \otimes a_n) \otimes_{kS_n} x)) \\
	&\quad=a \cdot \phi((a_0 \otimes \cdots \otimes a_n) 
	\otimes_{kS_{n+1}} x).
\end{align*}
	Hence, $\phi$ is an $A$-module homomorphism. 
	Finally, we construct a retraction of $\phi$. 
	We define the morphism 
	$\map{\psi}{A \otimes \sym{n-1}{}{A}}{\sym{n}{}{A}}$ by
$$
	\psi(a \otimes ((a_0 \otimes \cdots \otimes a_{n-1}) \otimes_{kS_n} x)) 
	= (a \otimes a_0 \otimes \cdots \otimes a_{n-1}) \otimes_{kS_{n+1}} x.
$$
	Since $S_n$ is a subgroup of $S_{n+1}$, 
	this morphism is well-defined. 
	Furthermore, we show that $\psi$ is an $A$-module homomorphism. 
	A direct calculation shows that, 
	for all $b,a, a_0, \dots , a_n \in A$ and $x \in k$,
	\begin{align*}
		\psi(b\cdot (a \otimes ((a_0 \otimes \cdots \otimes a_{n-1}) 
		\otimes_{kS_n} x))) 
		&= \psi(b^{(1)}a \otimes (b^{(2)} 
		\cdot (a_0 \otimes \cdots \otimes a_{n-1}) \otimes_{kS_n} x)) \\
		&= \psi(b^{(1)}a \otimes ((b^{(2)}a_0 \otimes \cdots 
		\otimes b^{(n+1)}a_{n-1}) \otimes_{kS_n} x)) \\
		&= (b^{(1)}a \otimes b^{(2)}a_0 \otimes \cdots 
		\otimes b^{(n+1)}a_{n-1}) \otimes_{kS_{n+1}} x \\
		&= b \cdot (a \otimes a_0 \otimes \cdots \otimes a_{n-1}) 
		\otimes_{kS_{n+1}} x \\
		&= b \cdot \psi(a \otimes ((a_0 \otimes \cdots \otimes a_{n-1}) 
		\otimes_{kS_n} x)).
	\end{align*}
	Hence, $\psi$ is an $A$-module homomorphism. 
	Moreover, for all $a_0,\dots ,a_n \in A$ and $x \in k$, we have
	\vspace{-0.5em}
\begin{align*}
	&\psi \circ \phi((a_0 \otimes \cdots \otimes a_n) 
	\otimes_{kS_{n+1}} x) \\
	&\quad= \psi \left( \frac{1}{n+1} \sum_{i = 0}^n (-1)^i a_i 
	\otimes ((a_0 \otimes \cdots \otimes a_{i-1} \otimes a_{i+1} \otimes \cdots 
	\otimes a_n) \otimes_{kS_n} x) \right) \\
	&\quad= \frac{1}{n+1} \sum_{i = 0}^n (-1)^i (a_i \otimes a_0 
	\otimes \cdots \otimes a_{i-1} \otimes a_{i+1} \otimes \cdots \otimes a_n) 
	\otimes_{kS_{n+1}}  x \\
	&\quad= \frac{1}{n+1} \sum_{i = 0}^n (-1)^i (a_i \otimes a_0 
	\otimes \cdots \otimes a_{i-1} \otimes a_{i+1} \otimes \cdots \otimes a_n) 
	\otimes_{kS_{n+1}}  (\sigma_1 \cdots \sigma_i) \cdot x \\
	&\quad= \frac{1}{n+1} \sum_{i = 0}^n (-1)^i (a_i \otimes a_0 
	\otimes \cdots \otimes a_{i-1} \otimes a_{i+1} \otimes \cdots \otimes a_n)
	\cdot (\sigma_1 \cdots \sigma_i) \otimes_{kS_{n+1}}  x \\
	&\quad= \frac{1}{n+1} \cdot (n+1) (a_0 \otimes \cdots \otimes a_n) 
	\otimes_{kS_{n+1}} x \\
	&\quad= (a_0 \otimes \cdots \otimes a_n) \otimes_{kS_{n+1}} x.
\end{align*}
	So, $\psi \circ \phi = \id{}$ holds. 
	While, by Lemma \ref{lem.hopf-proj}, 
	$A \otimes \sym{n-1}{}{A}$ is projective as a left $A$-module. 
	Therefore, $\sym{n}{}{A}$ is projective as a left $A$-module.
\end{proof}

\begin{rem}
	By Theorem \ref{thm.res-proj}, 
	if $\ch{k}\nmid (n+1)!$, then, for each $0 \leq m \leq n$, 
	there is an isomorphism $\hcoho{m}{A}{M} \cong \shcoho{m}{A}{M}$ 
	as $k$-vector spaces. 
	In particular, if $\ch{k} = 0$, then $\sym{\bullet}{}{A}$ 
	is a projective resolution of $k$, and hence there is an isomorphism 
	$\hcoho{\bullet}{A}{M} \cong \shcoho{\bullet}{A}{M}$ 
	as $k$-vector spaces.
\end{rem}
\subsection{Example}
In the last subsection, we describe an example 
of the resolution which gives symmetric cohomology. 
Let $p$ be a odd prime number, 
$k$ a field with characteristic $p$ and $C_p$ a cyclic group of order $p$. 
Then we calculate the symmetric cohomology of $A = kC_p$. 

\begin{prop}
	Let $p$ be a odd prime number, $\ch{k} = p$ and $A= kC_p$. 
	Then $\sym{n}{}{A}$ is a free $A$-module 
	with rank $\dfrac{{}_p C_{n+1}}{p}$ for each $1 \leq n \leq p-2$.
\end{prop}

\begin{proof}
	Let any generator 
	$(g_0 \otimes \cdots \otimes g_n) \otimes_{kS_{n+1}} 1$ 
	of $\sym{n}{}{A}$ as a left $A$-module. 
	There exists $0 \leq a_i \leq p-1$ such that $g_i$ is represented 
	by $g^{a_i}$ where $C_p = \langle g \mid g^p = e \rangle$. 
	We can change the order of tensor products 
	by an action of symmetric groups, 
	and hence we can assume $a_0 < a_1 < \cdots < a_n$ 
	without loss of generality. 
	First, we show 
    $$
	h \cdot ((g_0 \otimes \cdots \otimes g_n) \otimes_{kS_{n+1}} 1) 
	\neq \pm (g_0 \otimes \cdots \otimes g_n) \otimes_{kS_{n+1}} 1
    $$
for $(e\neq) h=g^a \in C_p$. 
Assume that there exists 
$h = g^a$ such that $h \cdot ((g_0 \otimes \cdots \otimes g_n) 
\otimes_{kS_{n+1}} 1) = \pm (g_0 \otimes \cdots \otimes g_n) 
\otimes_{kS_{n+1}} 1$. Then there exists $i$ such that
\begin{align*}
	a_j \equiv \begin{cases}
				a_{i+j} + a & (i+j \leq n), \\
				a_{i+j-(n+1)} + a & (i+j \geq n+1)
			\end{cases}
\end{align*}
modulo $p$ for $0 \leq j \leq n$. 
In particular, $a_0 \equiv a_i + a \pmod{p}$ and 
\begin{align*}
	a_i \equiv \begin{cases}
				a_{2i} + a & (2i \leq n), \\
				a_{2i-(n+1)} +a & (2i \geq n+1)
			\end{cases}
\end{align*}
hold. Hence, we have
\begin{align*}
	a_0 \equiv \begin{cases}
				a_{2i} + 2a & (2i \leq n), \\
				a_{2i-(n+1)} +2a & (2i \geq n+1).
			\end{cases}
\end{align*}
Repeating this operation, 
there exists $1 \leq m \leq n+1$ and $0 \leq l \leq m$ 
such that $mi-l(n+1)=0$, and hence $ma\equiv 0 \pmod{p}$ holds. 
Moreover, since $a_{mi-l(n+1)} = a_0$, we have $ma \equiv 0 \pmod{p}$. 
This result contradicts $1 \leq a \leq p-1$ 
and $1 \leq m \leq n+1 \leq p-1$. 
Therefore, 
$h \cdot ((g_0 \otimes \cdots \otimes g_n) 
\otimes_{kS_{n+1}} 1) \neq \pm (g_0 \otimes \cdots \otimes g_n) 
\otimes_{kS_{n+1}} 1$ for $(e\neq) h=g^a \in C_p$ is obtained.

Next, we show that generators of $\sym{n}{}{A}$ 
as a left $A$-module are not torsion. Assume that 
\begin{align*}
	a \cdot ((g_0 \otimes \cdots \otimes g_n) \otimes_{kS_{n+1}} 1)=0
\end{align*}
for any $a\in A$. If $a = \sum_{i=0}^{p-1} x_i g^i$, then
\begin{align*}
	a \cdot ((g_0 \otimes \cdots \otimes g_n) \otimes_{kS_{n+1}} 1) 
	= \sum_{i=0}^{p-1} x_i (g^i \cdot ((g_0 \otimes \cdots \otimes g_n) 
	\otimes_{kS_{n+1}} 1)).
\end{align*}
By the above discussion, 
each $g^i \cdot ((g_0 \otimes \cdots \otimes g_n) \otimes_{kS_{n+1}} 1)$ 
is linear independent. 
So, this implies $a = 0$, and hence generators are not torsion.

Finally, we construct a basis as an $A$-module from a basis 
as a $k$-vector space. 
Let $\mathcal{B} = \{ b_1,\dots ,b_{{}_pC_{n+1}} \}$ 
be a basis of $\sym{n}{}{A}$ as a $k$-vector space. 
We put $\mathcal{B}_{(i)}$ by remove $b_i$ from $\mathcal{B}$, 
if there exists $j$ such that $b_i \in \orb{C_p}{b_j} \cup \orb{C_p}{-b_j}$, 
where $\orb{C_p}{b}$ is an orbit of $b$ under the action of $C_p$. 
Moreover, continue this operation to $\mathcal{B}_{(i)}$. 
We construct the set $\mathcal{B}^\prime$ which has the property 
that intersections of orbits of each element are empty 
by repeating the operation. 
It is obtained that $\mathcal{B}^\prime$ is a set of generators 
as an $A$-module by its construction. 
Also, we have each element of $\mathcal{B}^\prime$ 
is linear independent because each element is not torsion, 
and intersections of orbits of each element are empty. 
Hence, $\sym{n}{}{A}$ is a free $A$-module. 
In particular, we get 
$\pm \mathcal{B} 
= \bigcup_{b\in \mathcal{B}^\prime} \left( \orb{C_p}{b} \cup 
\orb{C_p}{-b} \right)$ 
where we put $\pm \mathcal{B} = \{ \pm b \mid b\in \mathcal{B} \}$. 
While, we have $\left| \mathcal{B} \right| = {}_p C_{n+1}$ 
and $\left| \orb{C_p}{b} \right| = p$. 
Therefore, $\left| \mathcal{B}^\prime \right| = \dfrac{{}_pC_{n+1}}{p}$ holds, 
that is, the rank of $\sym{n}{}{A}$ is $\dfrac{{}_pC_{n+1}}{p}$.
\end{proof}
Since, $\sym{p-1}{}{A}$ is isomorphic to $k$ as a left $A$-module, 
the resolution of $k$ is the following exact sequence
    $$
	0 \rightarrow k \xrightarrow{\diff{p-1}{\widetilde{S}}} 
	\sym{p-2}{}{A} \rightarrow \cdots \rightarrow 
	\sym{1}{}{A} \xrightarrow{\diff{1}{\widetilde{S}}} 
	\sym{0}{}{A} \xrightarrow{\diff{0}{\widetilde{S}}} k \rightarrow 0,
	$$
where $\sym{i}{}{A}$ is a free $A$-module for each $0 \leq i \leq p-2$. 
This implies that there is an isomorphism 
$\hcoho{n}{A}{M} \cong \shcoho{n}{A}{M}$
for any left $A$-module $M$ and each $0 \leq n \leq p-2$. 
Also, in the case of $n=p-1$, the above isomorphism is obtained 
by simple calculation. Summarizing the above, we have 
	\begin{align*}
		\shcoho{n}{A}{M} \cong 
		   \begin{cases}
			 \hcoho{n}{A}{M} & (0 \leq n \leq p-1), \\
						   0 & (p \leq n).
		   \end{cases}
	\end{align*}




\begin{thebibliography}{99}
\bibitem{bardakov2018exterior}
V.~G.~Bardakov, M.~V.~Neshchadim and M.~Singh, 
\emph{Exterior and symmetric (co)homology of groups}, 
Internat. J. Algebra Comput. \textbf{30} (2020), no.8, 1577--1607. 
\bibitem{coconet2021symmetric}
T.~Coconet and C.-C.~Todea, 
\emph{Symmetric Hochschild cohomology of twisted group algebras}, 
arXiv:2103.13695 (2021). 
\bibitem{eilenberg1947cohomology}
S.~Eilenberg and S.~MacLane, 
\emph{Cohomology theory in abstract groups I}, 
Ann. of Math. \textbf{48}, 51--78, (1947). 
\bibitem{ginzburg1993cohomology}
V.~Ginzburg and S.~Kumar, 
\emph{Cohomology of quantum groups at roots of unity}, 
Duke Math. J. \textbf{69} (1993), no.1, 179--198. 
\bibitem{pirashvili2018symmetric}
M.~Pirashvili, 
\emph{Symmetric cohomology of groups}, 
J. Algebra \textbf{509} (2018), 397--418. 
\bibitem{singh2013symmetric}
M.~Singh, 
\emph{Symmetric continuous cohomology of topological groups}, 
Homology, Homotopy Appl. \textbf{15} (2013) 279--302. 
\bibitem{staic20093}
M.~D.~Staic, 
\emph{From $3$-algebras to $\Delta$-groups and symmetric cohomology}, 
J. Algebra \textbf{322} (2009), no.4, 1360--1378. 
\bibitem{staic2009symmetric}
M.~D.~Staic, 
\emph{Symmetric cohomology of groups in low dimension}, 
Arch. Math. \textbf{93} (2009), no.3, 205--211. 
\bibitem{swedler1969hopf}
M.~E.~Sweedler, 
\emph{Hopf Algebras}, 
Benjamin, New York, (1969). 
\bibitem{todea2015symmetric}
C.-C.~Todea, 
\emph{Symmetric cohomology of groups as a Mackey functor}, 
Bull. Belg. Math. Soc.
\textbf{22} (2015) 49--58.
\bibitem{witherspoon2019hochschild}
S.~J.~Witherspoon, 
\emph{Hochschild cohomology for algebras}, 
Graduate Studies in Mathematics, vol. 204, Amer. Math. Soc., (2019). 
\end{thebibliography}
\end{document}